\documentclass[a4paper,reqno, 12pt,titlepage]{amsart}
\usepackage{amssymb, graphicx, enumerate, enumitem, color}
\usepackage[usenames,dvipsnames,svgnames,table]{xcolor}
\usepackage[foot]{amsaddr}
\usepackage{comment}
\colorlet{mylinkcolor}{violet}
\colorlet{mycitecolor}{YellowOrange}
\colorlet{myurlcolor}{Aquamarine}
\usepackage{hyperref}
\usepackage{xcolor}
\hypersetup{
    pdftitle={Planar posets have dimension at most linear in their height},
    pdfauthor={Gwena\"el Joret, Piotr Micek, Veit Wiechert},
    colorlinks,
    linkcolor={red!50!black},
    citecolor={blue!50!black},
    urlcolor={blue!80!black}
}
\newtheorem{theorem}{Theorem}

\newtheorem{claim}[theorem]{Claim}

\newtheorem{lemma}[theorem]{Lemma}
\theoremstyle{remark}

\newcommand{\set}[1]{\{#1\}}

\newcommand{\calC}{\mathcal{C}}

\newcommand{\Z}{Z}

\DeclareMathOperator\Inc{Inc}
\DeclareMathOperator\cover{cover}

\DeclareMathOperator\Min{Min}
\DeclareMathOperator\Max{Max}
\DeclareMathOperator\Up{U}
\DeclareMathOperator\D{D}

\DeclareMathOperator\conv{conv}

\DeclareMathOperator\parent{parent}

\let\le\leqslant

\let\leq\leqslant
\let\geq\geqslant

\let\subset\subseteq

\let\epsilon\varepsilon

\let\region\calR
\DeclareMathOperator\m{m}
\DeclareMathOperator\redv{red_v}
\DeclareMathOperator\bluev{blue_v}
\DeclareMathOperator\rede{red_e}
\DeclareMathOperator\bluee{blue_e}
\DeclareMathOperator\Red{Red}
\DeclareMathOperator\Blue{Blue}
\newcommand{\RED}{X'}
\newcommand{\BLUE}{Y'}

 \renewenvironment{enumerate}{\begin{enumorig}[label=\textup{(\roman*)}, noitemsep, topsep=2pt plus 2pt, labelindent=.2em, leftmargin=*, widest=iii]}{\end{enumorig}}

\newenvironment{enumerateAlpha}{\begin{enumorig}[label=\textup{(\alph*)}, noitemsep, topsep=2pt plus 2pt, labelindent=.2em, leftmargin=*, widest=iii]}{\end{enumorig}}

\makeatletter
\let\old@setaddresses\@setaddresses
\def\@setaddresses{\bigskip\bgroup\parindent 0pt\let\scshape\relax\old@setaddresses\egroup}
\makeatother

\begin{document}
\title{Planar posets have dimension at most linear in their height}

\author[G.~Joret]{Gwena\"{e}l Joret}
\address[G.~Joret]{Computer Science Department \\
  Universit\'e Libre de Bruxelles, 
  Brussels, 
  Belgium}
\email{gjoret@ulb.ac.be}

\author[P.~Micek]{Piotr Micek}
\address[P.~Micek]{Theoretical Computer Science Department\\
  Faculty of Mathematics and Computer Science, Jagiellonian University, Krak\'ow, Poland \\
and Institute of Mathematics, Combinatorics and Graph Theory Group \\
Freie Universit\"at Berlin, Berlin, Germany}
\email{piotr.micek@tcs.uj.edu.pl}

\author[V.~Wiechert]{Veit Wiechert}
\address[V.~Wiechert]{Institut f\"ur Mathematik\\
  Technische Universit\"at Berlin, 
  Berlin, 
  Germany}
\email{wiechert@math.tu-berlin.de}

\thanks{G.\ Joret is supported by an ARC grant from the Wallonia-Brussels Federation of Belgium. 
        P.\ Micek is partially supported by a Polish National Science Center grant (SONATA BIS 5; UMO-2015/18/E/ST6/00299).  
	V.\ Wiechert is supported by the Deutsche Forschungsgemeinschaft within the research training group `Methods for Discrete Structures' (GRK 1408).}

\date{\today}



\begin{abstract} 
We prove that every planar poset $P$ of height $h$ has dimension at most $192h+96$.  
This improves on previous exponential bounds and is best possible up to a constant factor.
We complement this result with a construction of planar posets of height $h$ and dimension at least $(4/3)h-2$. 
\end{abstract}
\maketitle

\section{Introduction}
In this paper we study finite partially ordered sets (posets for short) whose diagram can be drawn in a planar way, called \emph{planar posets}.
Unlike planar graphs, planar posets have a rather wild structure. For instance, recognizing planar posets is an NP-complete problem. More importantly for the purpose of this paper, planar posets have unbounded dimension, contrasting with the fact that planar graphs are $4$-colorable.
Recall that the \emph{dimension} of a poset $P$, which we denote by $\dim(P)$, is defined as the least integer $d$ such that $P$ is the intersection of $d$ linear orders.
This invariant is a central measure of a poset's complexity, playing a role similar to that of the chromatic number for graphs.

It is well known that for a poset to have large dimension, the poset must be {\em wide}: By an old theorem of Dilworth, dimension is bounded from above by the width of the poset (the maximum size of an antichain).
A remarkable feature of planar posets is that if they have large dimension then they are also {\em tall}: By a recent theorem of Streib and Trotter~\cite{ST14}, dimension is bounded from above by a function of the height of the poset (the maximum size of a chain).   
The main contribution of this paper is that planar posets have dimension at most linear in their height:

\begin{theorem}\label{thm:main}
	If $P$ is a planar poset of height $h$ then
	\[
	\dim(P) \leq 192h+96.
	\]
\end{theorem}

The previous best known bound was exponential in the height.
A linear bound is optimal up to a constant factor, as shown by a construction of Kelly~\cite{Kel81} from 1981 that has dimension $h+1$.
A well-known variant of this construction, with dimension $h-1$, is illustrated in Figure~\ref{fig:kelly}.
These examples still provide the best known lower bound on the maximum dimension of a planar poset of height $h$.  
Our second contribution is a slight improvement of this lower bound:
\begin{theorem}\label{thm:lower-bound-1}
	For every $h\geq 1$, there is a planar poset $P$ of height $h$ with
	\[
	\dim(P) \geq (4/3)h-2.
	\]
\end{theorem}

The upper bound of Streib and Trotter~\cite{ST14} on dimension in terms of height holds in fact in a more general setting, that of posets with planar cover graphs.  
Informally, the {\em cover graph} of a poset is just its diagram seen as an undirected graph.
It is not known whether the dimension of posets with planar cover graphs is bounded by a linear function of their height (or any polynomial function for that matter).    
We present a slightly better construction in that less restrictive setting:
\begin{theorem}\label{thm:lower-bound-2}
	For every $h\geq 1$, there is a poset of height $h$ with a planar cover graph and dimension at least $2h-2$.
\end{theorem}

\begin{figure}[t]
	\centering
	\includegraphics[scale=1.0]{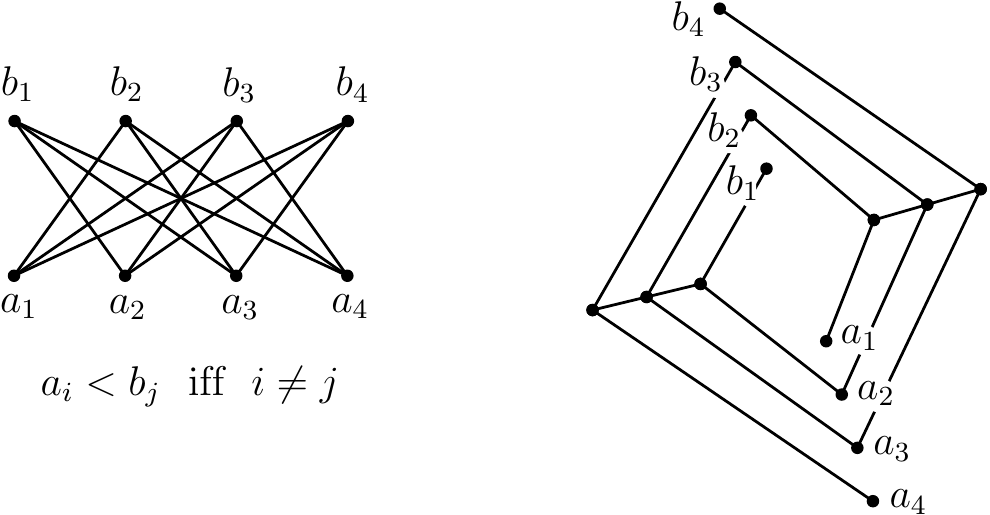}
	\caption{\label{fig:kelly} Standard example $S_4$ (left) and Kelly's construction containing $S_4$ (right).}
\end{figure}

It seems that planarity plays a crucial role in the existence of small bounds on the dimension in terms of the height.    
Indeed, our last contribution is a construction of posets having no $K_5$ minor in their cover graphs and with dimension exponential in the height.

\begin{theorem}\label{thm:lower-bound-K5}
	For each even $h\geq 2$, there is a poset of height $h$ with dimension at least $2^{h/2}$ whose cover graph excludes $K_5$ as a minor.
\end{theorem}

In order to motivate our results, we provide background on planar posets, their dimension, and related problems in the rest of this introduction.

Let us start by mentioning the following equivalent definition of dimension:  $\dim(P)$ is the least integer $d$ such that the elements of $P$ can be embedded into $\mathbb{R}^d$ in such a way that $x<y$ in $P$ if and only if the point of $x$ is below the point of $y$ with respect to the product order on $\mathbb{R}^d$.
One of the very first results about dimension is that planar posets having a single minimal element and a single maximal element can be embedded in the plane:
\begin{theorem}[Baker, Fishburn, and Roberts~\cite{BFR72}]\label{thm:zero-one}
	For every planar poset $P$ with a single minimal element and a single maximal element, we have
	\[
	\dim(P)\leq 2.
	\]
\end{theorem}

Similarly, planar posets with a single minimal element can be embedded in $\mathbb{R}^3$:

\begin{theorem}[Trotter and Moore~\cite{TM77}]\label{thm:zero-or-one}
	For every planar poset $P$ with a single minimal element, 
	\[
	\dim(P)\leq 3.
	\]
\end{theorem}

Yet, in general planar posets can have arbitrarily large dimension.
This was first shown by Kelly~\cite{Kel81} in 1981.  
Figure~\ref{fig:kelly} (right) illustrates Kelly's construction\footnote{more precisely, a simpler variant of his original construction that is standard in the literature} of order $4$; its definition for any fixed order $k \geq 2$ can be inferred from the figure.  
This poset contains as an induced subposet the height-$2$ poset depicted on the left, the so-called standard example $S_4$.
For $k \geq 2$, the \emph{standard example} $S_k$ is defined as the poset consisting of $k$ minimal elements $a_1,\ldots,a_k$ and $k$ maximal elements $b_1,\ldots,b_k$, such that $a_i<b_j$ in $S_k$ if and only if $i\neq j$.
It is not hard to see that $S_k$ has dimension exactly $k$, for all $k\geq 2$.
Hence, Kelly's construction in Figure~\ref{fig:kelly} has dimension at least $4$, and in general Kelly's examples have arbitrarily large dimension.

We note that there is a variant of Kelly's construction known as the {\em spider net}~\cite{ST14}, which is not a planar poset but still has a planar cover graph, and has a single minimal element, a single maximal element, and unbounded dimension (see Figure~\ref{fig:trotter-spider}).  
This is in sharp contrast with Theorem~\ref{thm:zero-one} for planar posets.

\begin{figure}[t]
	\centering
	\includegraphics[scale=1.0]{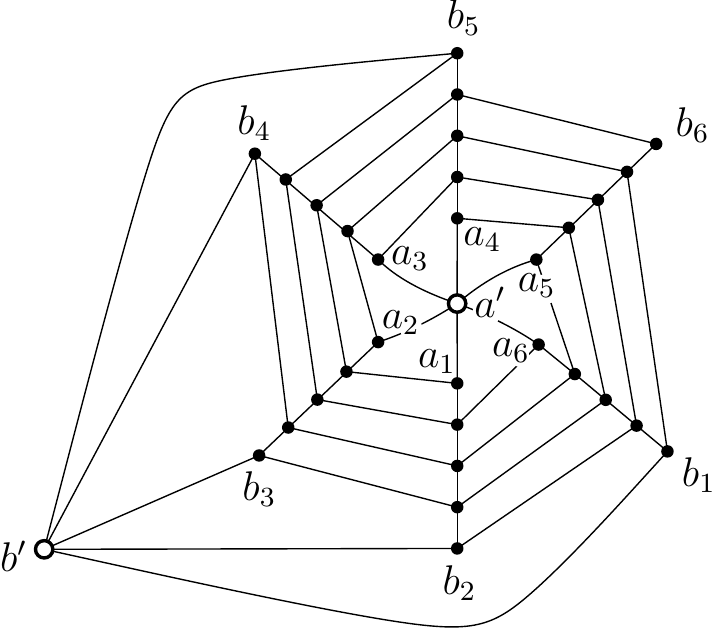}
	\caption{\label{fig:trotter-spider}The spider net of order $6$ (from~\cite{ST14}), a poset with a planar cover graph, a single minimal element ($a'$) and a single maximal element ($b'$). Taking $a'$ as center, cover relations are oriented towards the exterior.}
\end{figure}

While there are posets of bounded height with unbounded dimension (such as the standard examples), it is worth noting that the height of Kelly's construction grows with its order.
As mentioned earlier, this is not a coincidence:

\begin{theorem}[Streib and Trotter~\cite{ST14}]\label{thm:ST}
	There is a function $f:\mathbb{N}\to \mathbb{N}$ such that for every $h\geq 1$ and every poset $P$ of height $h$ with a planar cover graph, 
	\[
	\dim(P) \leq f(h).
	\]
\end{theorem}

It soon turned out that a bound on the dimension in terms of the height holds more generally for posets with a sparse and `well-structured' cover graph.
This was shown in a series of generalizations of Theorem~\ref{thm:ST}:	 
Given a class of graphs $\calC$, every poset with a cover graph in $\calC$ has dimension bounded in terms of its height if
\begin{enumerate}
	\item $\calC$ excludes an apex graph as a minor (\cite{JMMTWW});
	\item $\calC$ excludes a fixed graph as a (topological) minor (\cite{Walczak17,MW});
	\item $\calC$ has bounded expansion (\cite{JMW-sparsity}).
\end{enumerate}

It is worth noting that in all the cases above the resulting bounds on the dimension are exponential in the height. 
As mentioned before, this is unavoidable since there are posets with $K_5$-minor-free cover graphs having dimension exponential in their height (Theorem~\ref{thm:lower-bound-K5}). 

The paper is organized as follows. 
First, we provide some preliminaries on dimension in Section~\ref{sec:preliminaries}. 
Then we prove Lemma~\ref{lemma:point-below-max} in Section~\ref{sec:main-lemma}, a lemma that is at the heart of our proof of Theorem~\ref{thm:main}. 
Informally, this lemma states that a linear bound holds for planar posets that have a special element $x_0$ such that $x_0 \leq b$ for all maximal elements $b$ of the poset.      
Its proof uses heavily the planarity of the diagram (as opposed to merely using the planarity of the cover graph). 
Next, we introduce some extra tools in Section~\ref{sec:tools}, including a standard technique called `unfolding a poset' that originates in the work of Streib and Trotter~\cite{ST14}. 
We finish the proof of our main theorem in Section~\ref{sec:main-proof}, by reducing the case of general planar posets to the case covered by  Lemma~\ref{lemma:point-below-max}. 
Finally, we describe in Section~\ref{sec:lower-bounds} the various constructions mentioned at the beginning of the introduction.

\section{Preliminaries}\label{sec:preliminaries}
Let us first recall some basic definitions.
Let $P$ be a poset.
A {\em chain} in $P$ is a set of pairwise comparable elements of $P$.
The {\em height} of $P$ is the maximum size of a chain in $P$. 
We let $\Min(P)$ and $\Max(P)$ denote the set of minimal and maximal elements of $P$, respectively.

A relation $a \leq b$ in $P$ is a {\em cover relation} of $P$ if $a \neq b$ and $a \leq b$ is not implied by transitivity, that is, there is no element $c$ distinct from $a, b$ such that $a \leq c \leq b$ in $P$.  
The usual way to visualize $P$ is to draw its \emph{diagram}:  
Elements are drawn as distinct points in the plane and each cover relation $a \leq b$  in $P$ is represented by a curve from $a$ to $b$ going upwards.  
If this diagram can be drawn in a planar way then $P$ is said to be a {\em planar poset}.   
The {\em cover graph} of $P$, denoted $\cover(P)$, is the undirected graph with the set of elements of $P$ as vertex set, and where two elements $a, b$ are adjacent if they are in a cover relation in $P$.
(Informally, this is the diagram of $P$ seen as an undirected graph.) 

While planar posets have a planar cover graph, the converse is not necessarily true: 
For every $h\geq 3$, there is a non-planar poset of height $h$ having a planar cover graph.
Two examples of height $4$ are shown in Figure~\ref{fig:non-planar}. 
(In the special case of height-$2$ posets however, the two properties are equivalent~\cite{Moore-thesis}.)  

\begin{figure}[t]
	\centering
	\includegraphics[scale=1.0]{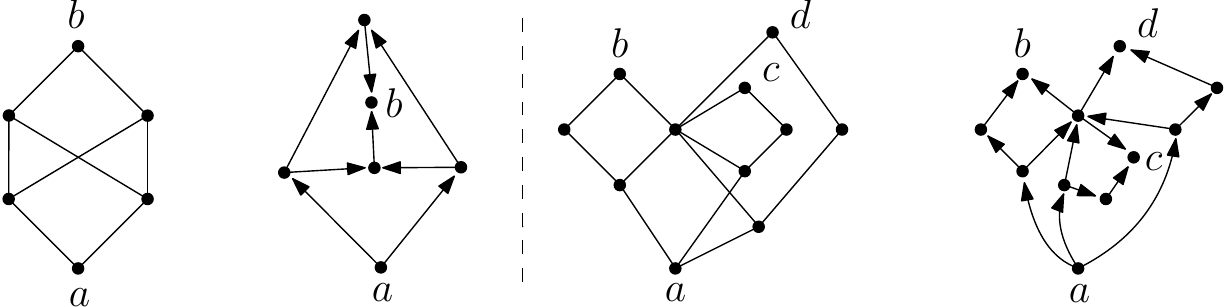}
	\caption{\label{fig:non-planar}
		Posets with non-planar diagrams but planar cover graphs.}
\end{figure}

A poset $Q$ is an {\em induced subposet} of $P$ if $Q$ is obtained by selecting a subset $X$ of the elements of $P$ together with all their relations, that is, given $a,b\in X$, we have $a\leq b$ in $Q$ if and only if $a \leq b$ in $P$.

The {\em upset $\Up_P(x)$} of an element $x\in P$ is the set of all elements $y\in P$ such that $x \leq y$ in $P$. 
If we reverse the relation $x \leq y$ into $x \geq y$, we get the definition of the {\em downset} of $x$, which we denote by $\D_P(x)$. 
In both cases, we drop the subscript $P$ when the poset is clear from the context.   

An {\em incomparable pair} of $P$ is an ordered pair $(x,y)$ of elements of $P$ that are incomparable in $P$.
We denote by $\Inc(P)$ the set of incomparable pairs of $P$.
Let $I \subseteq \Inc(P)$ be a non-empty set of incomparable pairs of $P$.
We say that $I$ is \emph{reversible} if there is a linear extension $L$ of $P$ \emph{reversing} each pair of $I$, that is, we have $x>y$ in $L$ for every $(x,y)\in I$.
By $\dim(I)$ we denote the least $d$ such that $I$ can be partitioned into $d$ reversible sets. 
We will use the convention that $\dim(I)=1$ when $I$ is an empty set.   
It is easily seen that $\dim(P)$ is equal to $\dim(\Inc(P))$.
Given sets $A,B\subseteq P$, we let $\Inc(A,B)$ be the set of pairs $(a,b)\in\Inc(P)$ with $a\in A$ and $b\in B$.
We use the abbreviation $\dim(A,B)$ for $\dim(\Inc(A,B))$.

A sequence $(x_1,y_1), \dots, (x_k,y_k)$ of pairs from $\Inc(P)$ with $k \geq 2$ is said to be
an \emph{alternating cycle of size $k$} if $x_i\leq y_{i+1}$ in $P$ for all $i\in\set{1,\ldots,k}$ (cyclically, so $x_k\le y_1$ in $P$ is required).
We call it a \emph{strict} alternating cycle if for each $i,j\in\set{1,\ldots,k}$, we have $x_i\le y_{j}$ in $P$ if and only if $j=i+1$ (cyclically).
Note that in this case $x_1, x_2, \dots, x_k$ are all distinct, and $y_1, y_2, \dots, y_k$ are all distinct.
(We remark that we could have $x_i=y_{i+1}$ for some $i$'s.)

Observe that if $(x_1,y_1), \dots, (x_k,y_k)$ is a strict alternating cycle in $P$, then this set of incomparable pairs cannot be reversed by a linear extension $L$ of $P$.
Indeed, otherwise we would have $y_i < x_i \leq y_{i+1}$ in $L$ for each $i \in \{1,2, \dots, k\}$ cyclically, which cannot hold.
Hence, strict alternating cycles are not reversible.
The converse is also true and this was originally observed
by Trotter and Moore~\cite{TM77}:
A set $I$ of incomparable pairs of a poset $P$ is reversible if and only if $I$ contains no strict alternating cycle.

The following easy lemma allows us to restrict our attention to incomparable pairs of the form $(a,b)$ where $a$ is a minimal element and $b$ is a  maximal element (see for instance~\cite[Observation~3]{JMTWW} for a proof).

\begin{lemma}\label{lem:min-max-reduction}
For every poset $P$ there is a poset $Q$ containing $P$ as an induced subposet such that $P$ and $Q$ have the same height and 
\begin{enumerate}
\item $\cover(Q)$ is obtained from $\cover(P)$ by adding some degree-$1$ vertices to $\cover(P)$, and\label{item:deg1}
 \item $\dim(P)\leq \dim(\Min(Q), \Max(Q))$.
\end{enumerate}
\end{lemma}

For our purposes, an important observation is that if $P$ has a planar diagram then it is easily seen from property~\ref{item:deg1} that the poset $Q$ in the above lemma also has a planar diagram.
Hence, in order to bound $\dim(P)$ by a function of the height of $P$, it is enough to bound $\dim(\Min(Q),\Max(Q))$ in terms of the height of $Q$ (which equals that of $P$).
This will be used in the proof of our main theorem.
We will refer to $\dim(\Min(Q),\Max(Q))$ as the \emph{min-max dimension} of $Q$.

\section{Main lemma}\label{sec:main-lemma}

In this section we prove the following lemma, which is the key lemma in the proof of our main theorem. 
We remark that the proof of the lemma uses heavily the drawing of the diagram in the plane and does not extend to the case of a poset with a planar cover graph.

\begin{lemma}\label{lemma:point-below-max}
 Let $P$ be a planar poset of height $h$ and let $B\subseteq \Max(P)$.
 If there is an element $x_0$ such that $x_0\leq b$ in $P$ for all $b\in B$, then $\dim(\Min(P),B)\leq 6h+3$.
\end{lemma}

The proof of Lemma~\ref{lemma:point-below-max} is split into a number of steps.  
The goal is to partition the set $\Inc(\Min(P),B)$ into at most $6h+3$ reversible subsets.
Here is a brief outline of the proof: 
First, we deal with some incomparable pairs that can easily be reversed using a bounded number of linear extensions (three). 
The remaining set of incomparable pairs has a natural partition into those pairs $(a,b)$ such 
that $a$ is `to the left' of $b$ and those such that $a$ 
is `to the right' of $b$ (see Claim~\ref{claim:left-and-right}). 
Using symmetry, we can then focus on one of these two types of incomparable pairs.   
At that point, we study how alternating cycles must look like.   
We establish various technical properties of these alternating cycles, which enable us to show that if the dimension is large, then there is a large `nested structure' in the diagram of $P$, with some maximal element $b$ buried deep inside and the special element $x_0$ drawn outside the structure. 
This is where the existence of $x_0$ is used: A path in the diagram witnessing the relation $x_0 \leq b$ in $P$ has to go through the whole nested structure. 
This path will then meet a number of disjoint curves from the diagram, and by planarity the path can only enter such a curve in an element of the poset. 
This way, we deduce that the path under consideration contains many elements from $P$, and 
therefore that the height of $P$ is large. 

Let us now turn to the proof of Lemma~\ref{lemma:point-below-max}. 
Fix a planar drawing of the diagram of $P$. 
Let $A:=\Min(P)$ and $I_1:=\Inc(A,B)$. 
We may assume without loss of generality that $x_0 \in A$. 
(Indeed, if not then there exists $a\in A$ with $a < x_0$ in $P$ and we simply take $a$ as the new $x_0$.) 
We use the standard coordinate system where each point in the plane is characterized by an $x$-coordinate and a $y$-coordinate.
We may assume without loss of generality that no two elements of $P$ have the same $y$-coordinate in the drawing.

\subsection{Some easy cases}
Given two distinct elements $a,b \in P$, we say that $a$ is \emph{drawn below} 
(\emph{above}) $b$
if the $y$-coordinate of $a$ is less than (greater than, respectively) that of $b$.

Let $I'_1$ be the set of all incomparable pairs $(a,b)\in I_1$ such that $a$ is drawn above $b$.
Observe that if we order the elements of $P$ by increasing order of their 
$y$-coordinates in the drawing, then we obtain a linear extension of $P$ (by the 
definition of a diagram).
This linear extension reverses all pairs in $I'_1$.
Let $I_2:=I_1 - I'_1$ be the set of remaining incomparable pairs in $I_1$, that is, 
those pairs $(a,b) \in I_1$ such that $a$ is drawn below $b$.  
Clearly, 
\[
\dim(I_1) \leq \dim(I_2) + \dim(I_1') \leq \dim(I_2) + 1,
\]
thus we can restrict our attention to pairs in $I_2$.  

Before pursuing further, let us introduce some terminology. 
A {\em witnessing path} for a relation $a\leq b$ in $P$ is a path $u_1,u_2, \dots ,u_k$ in the cover 
graph of $P$ with $u_1=a$, $u_k=b$, and $u_i \leq u_{i+1}$ being a cover relation of $P$ for each 
$i\in \{1, \dots, k-1\}$. 
A curve in the plane that can be oriented so that it is strictly increasing on the 
$y$-coordinate is said to be {\em $y$-increasing}.
If a $y$-increasing curve $\gamma$ is completely contained in the drawing of $P$ then we call 
$\gamma$ a {\em walk} in the diagram of $P$.  
Every witnessing path for $a\leq b$ in $P$ corresponds to a walk from $a$ to $b$.
Conversely, every two elements of $P$ contained in the same walk are comparable in $P$. 
For an element $p$ of $P$, we define the \emph{$p$-line} as the horizontal line in the plane through (the image of) $p$.
When $p$ and $q$ are two elements of $P$, we say that $p$ \emph{sees} the $q$-line if there is a walk containing $p$ that intersects the $q$-line.

Let $I'_2$ be the set of all pairs $(a,b)\in I_2$ such that $a$ does not see the $b$-line. 
It turns out that the set $I_2'$ is reversible:  
If not, then $I_2'$ contains an alternating cycle $(a_1,b_1), \dots, (a_k,b_k)$.
We may assume that $b_1$ is drawn below all other $b_i$'s.
Given that $a_1$ is drawn below $b_1$ (since $(a_1,b_1)\in I_2$) and $b_2$ is drawn above $b_1$, 
any witnessing path for the relation $a_1\leq b_2$ in $P$ crosses the $b_1$-line. 
Thus $a_1$ sees the $b_1$-line, contradicting $(a_1,b_1)\in I_2'$.

Let $I_2''$ be a set of all pairs $(a,b)\in I_2$ such that $b$ does not see the $a$-line.
A dual argument to the one above shows that $I_2''$ is reversible.
Let $I_3:=I_2-(I'_2 \cup I''_2)$, that is, $I_3$ is the set of pairs $(a,b)\in I_2$ such that 
$a$ sees the $b$-line and $b$ sees the $a$-line.
We have
\[
\dim(I_1) \leq \dim(I_2) + 1 \leq \dim(I_3) + \dim(I_2') + \dim(I_2'') + 1 \leq \dim(I_3) +3.
\]
Thus, to prove Lemma~\ref{lemma:point-below-max}, it remains to partition $I_3$ into at most $6h$ reversible sets.

Given two distinct elements $p,q\in P$, we say that $p$ \emph{sees the left side (right side)} of 
$q$ if there is a walk containing $p$ and intersecting the $q$-line to the left (right) of $q$.
Note that in general $p$ could possibly see both the left and the right side of $q$.
However, this cannot happen with elements of pairs in $I_3$:

\begin{claim}\label{claim:left-and-right}
 For each $(a,b)\in I_3$, either
 \begin{enumerate}
  \item $a$ sees only the left side of $b$ and $b$ sees only the right side of $a$, or\label{item:left-right}
  \item $a$ sees only the right side of $b$ and $b$ sees only the left side of $a$.\label{item:right-left}
 \end{enumerate}
\end{claim}
\begin{figure}[h]
	\centering
	\includegraphics[scale=1.0]{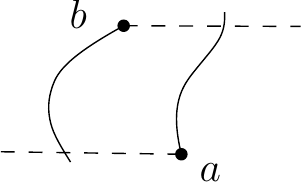}
	\caption{\label{fig:claim-left-and-right}An incomparable pair $(a,b)$ satisfying Claim~\ref{claim:left-and-right}\ref{item:right-left}.}
\end{figure}
\begin{proof}
Let  $(a,b)\in I_3$.
Thus $a$ sees the $b$-line and $b$ sees the $a$-line. 
We may assume without loss of generality that $a$ sees the right side of $b$ (see Figure~\ref{fig:claim-left-and-right}), the other case being symmetric.
Let $\gamma_a$ be a walk witnessing the fact that $a$ sees the right side of $b$.
First, we show that $b$ cannot see the right side of $a$.
Suppose for a contradiction that it does, and let $\gamma_b$ be a walk witnessing this.
Clearly, $\gamma_a$ and $\gamma_b$ intersect.
Since the drawing of the diagram of $P$ under consideration is planar, the 
intersection of  $\gamma_a$ and $\gamma_b$ contains the image of
an element $p\in P$. 
Since there is a walk containing $a$ and $p$, we see that $a\leq p$ in $P$, and similarly  $p\leq b$ in $P$ since there is a walk containing $p$ and $b$.
We conclude that $a\leq b$ in $P$,  a contradiction.
Thus, $b$ cannot see the right side of $a$, and hence only sees its left side.

Now, if $a$ sees both the right side and the left side of $b$, we conclude again that $a\leq b$ in $P$ with an analogue argument, which is a contradiction. 
Hence, $a$ sees only the right side of $b$. 
This completes the proof of the claim.
\end{proof}

We partition the set of pairs $(a,b) \in I_3$ into two sets $I'_3$ and $I''_3$, depending on 
whether $(a,b)$ satisfies~\ref{item:left-right} or~\ref{item:right-left} in the claim above. 
It is enough to show that we can partition one of the two sets into at most $3h$ reversible sets, as we would obtain the same result for the other set by symmetric arguments (that is, by exchanging the notion of left and right).
We focus on the set $I''_3$, and thus aim to prove that $\dim(I''_3)\leq 3h$.
For convenience, let $I_4 := I''_3$.

\begin{figure}[b]
	\centering
	\includegraphics[scale=1.0]{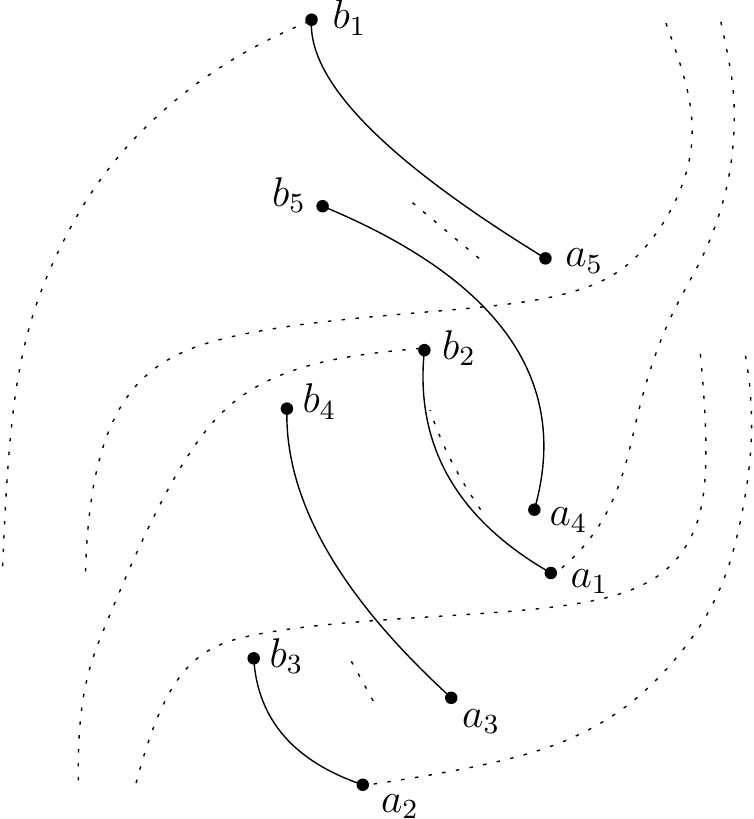}
	\caption{\label{fig:alternating-cycle-example}An alternating cycle of size $5$ with all pairs in $I_4^{\text{sep}}$.
	Solid lines are walks witnessing comparabilities of the cycle.
	Dotted lines are walks witnessing that for all $i$: (1) $a_i$ sees the right side of $b_i$; (2) $b_i$ sees the left side of $a_i$; (3) the pair $(a_i,b_i)$ is separated.}
\end{figure}

We are going to partition pairs in $I_4$ according to whether they admit a `separator': Say that a walk $\sigma$ is a \emph{separator} for a pair $(a,b) \in I_4$ if $\sigma$ starts on the $a$-line to the left of $a$, ends on the $b$-line to the right of $b$, and every element of $P$ that appears on $\sigma$ is incomparable to both $a$ and $b$. 
Then let $I_4^{\text{sep}}$ be the set of pairs in $I_4$ having a separator, and let  $I_4^{\text{no-sep}}$ be the set of pairs in $I_4$ with no separator. 
See Figure~\ref{fig:alternating-cycle-example} for an example of an alternating cycle with all pairs in $I_4^{\text{sep}}$. 

Say that an incomparable pair $(a,b) \in I_4$ is {\em dangerous} if $a$ is drawn below $x_0$ and $a$ sees the left side of $x_0$. 
Note if $(a,b)$ is {\em not} dangerous, then either $a$ is drawn above $x_0$, or $a$ is drawn below $x_0$ but then $a$ only sees the right side of $x_0$ (recall that there is walk from $a$ to the $b$-line). 

We consider dangerous and non-dangerous pairs in $I_4^{\text{sep}}$ and $I_4^{\text{no-sep}}$.  
This defines four subsets of incomparable pairs. 
A first observation is that one of these sets is empty:   

\begin{claim}
There are no dangerous pairs in $I_{4}^{\text{sep}}$. 
\end{claim}
\begin{proof}
Arguing by contradiction, suppose that $(a,b)\in I_{4}^{\text{sep}}$ is dangerous. 
Let $\sigma$ denote a walk that separates the pair. 
Thus $a$ is to the right of $\sigma$ and $b$ is to its left. 
Since $a$ is drawn below $x_0$, the walk $\sigma$ intersects the $x_0$-line. 
Consider a walk $\beta$ from $a$ to the $x_0$-line with $x_0$ to its right, which exists since $(a,b)$ is dangerous. 
By the definition of a separator, $\beta$ cannot intersect $\sigma$, and hence the top endpoint of $\beta$ is to right of $\sigma$. 
It follows that $x_0$ is to the right of $\sigma$. 
However, this implies that any walk witnessing the relation $x_0 \leq b$ in $P$ must intersect $\sigma$, since $b$ is to the left of $\sigma$. 
This contradicts the fact that $\sigma$ separates the pair $(a,b)$.   
\end{proof}

The plan for the rest of the proof is to partition into at most $h$ reversible sets each of the remaining three sets of incomparable pairs, i.e.\ non-dangerous pairs in $I_4^{\text{sep}}$, non-dangerous pairs in $I_4^{\text{no-sep}}$, and dangerous pairs in $I_4^{\text{no-sep}}$. 
Altogether, this proves that $\dim(I_4) \leq 3h$, as desired. 
To do so, we first make a little detour: In the next few pages we introduce the key notion of a `wall' and establish several useful properties of walls. 
These walls will help us getting a better understanding of strict alternating cycles in $I_4^{\text{sep}}$ and in $I_4^{\text{no-sep}}$, which in turn will help us to partition the three sets mentioned above into at most $h$ reversible sets each. 

\subsection{Walls}

Given a point $p$ in the plane and a walk $\gamma$ that intersects the $p$-line in a point $q$ 
distinct from $p$, we say that $p$ is {\em to the left (right) of $\gamma$} if $p$ is to the left 
(right, respectively) of $q$ on the $p$-line. 
A set $W$ of walks is a \emph{wall} if there is a walk $\gamma \in W$ such that every walk $\gamma' \in W$ distinct from $\gamma$ has the property that its topmost point is to the right of some walk in $W$.
Note that in this case the walk $\gamma$ is uniquely defined; we call it the {\em root} of the wall $W$.
Observe also that the topmost point of the root walk has the maximum $y$-coordinate among all points in walks of $W$.
We begin by showing an easy property of walls.

\begin{claim}\label{claim:topology-of-the-wall}
Let $W$ be a wall and let $\gamma$ be a walk that is disjoint from every walk in $W$. 
If the bottommost point of $\gamma$ is to the right of some walk in $W$,
then every point of $\gamma$ is either to the right of some walk in $W$,
or above all the walks in $W$.
\end{claim}
\begin{figure}[h]
	\centering
	\includegraphics[scale=1.0]{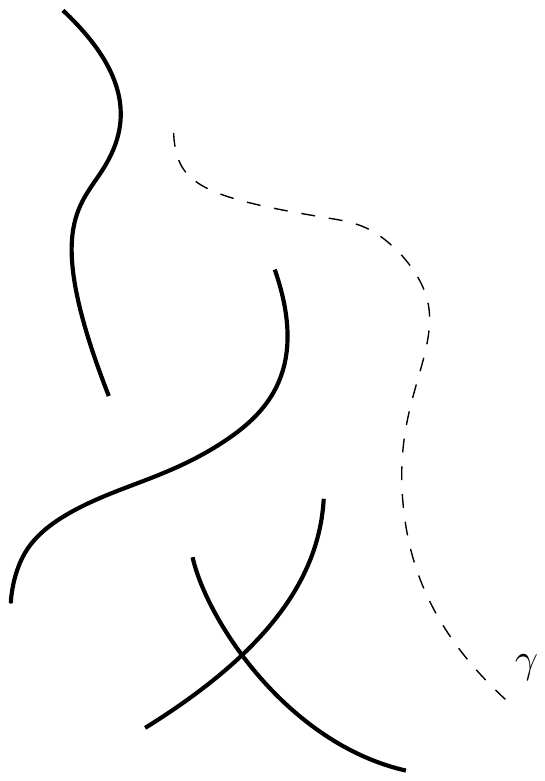}
	\caption{\label{fig:wall-and-a-curve} A wall (thick curves) and a walk $\gamma$ that is disjoint from every walk in the wall.}
\end{figure}

\begin{proof}
Assume that the bottommost point of $\gamma$ is to the right of some walk in $W$.
We are going to walk along $\gamma$ starting at its bottommost point and going upwards. 
We claim that at all times the current point satisfies the property that it is either to the right of some walk in $W$, or above all the walks in $W$.
This is true at the beginning, as by our assumption the bottommost point of $\gamma$ is to the right of some walk, say $\delta$, in $W$.
Since $\gamma$ does not intersect $\delta$, walking along $\gamma$ we have the curve $\delta$ to 
the left, until one of the two curves ends. 
If $\gamma$ stops first then every point of $\gamma$ is to the right of some walk in $W$, as desired.
If $\delta$ stops first, then either $\delta$ is the root walk of $W$ and therefore all points of $\gamma$ above the current point are above all walks in $W$, 
or $\delta$ is not the root walk and by definition of a wall there is another walk $\delta'$ in 
$W$ to the left of $\delta$'s topmost point, and hence to the left of the current point of $\gamma$.
Continuing in this way, we see that every point of $\gamma$ has the desired property.
This completes the proof.
\end{proof}

Consider a strict alternating cycle $C$ in $I_4$ consisting of the pairs $(a_1,b_1), \dots, (a_k,b_k)$ such that $b_1$ is drawn above all other $b_i$'s. 
For $i \in \{2, \dots, k\}$, we say that a wall $W$ is an \emph{$i$-wall for the cycle $C$} if 
for each element $p \in P$ that is included in some walk of $W$, there exists $\ell \in \{1, \dots, i-1\}$ such that $a_{\ell} \leq p$ in $P$ or $p \leq b_{\ell+1}$ in $P$.
Note that if $W$ is an $i$-wall for $C$ then $W$ is also a $j$-wall for $C$, for every $j\in\{i+1,\ldots,k\}$. 

Say that a point $p$ is {\em to the left of a wall $W$} if $p$ is to the left of all walks in $W$ intersecting the $p$-line, and there is at least one such walk.

\begin{claim}\label{claim:i-wall}
Let $C$ denote a strict alternating cycle $(a_1,b_1), \dots, (a_k,b_k)$ in $I_4$ such that $b_1$ is drawn topmost among the $b_i$'s.
If $W$ is an $i$-wall for $C$ for some $i \in \{2, \dots, k\}$ and $b_1$ is to the left of $W$, then
none of $a_i, \dots, a_k$, $b_{i}, \dots, b_k$ is to the right of some walk in $W$ (see Figure~\ref{fig:wrong-alternating-cycles}).
\end{claim}
\begin{figure}[h]
	\centering
	\includegraphics[scale=1.0]{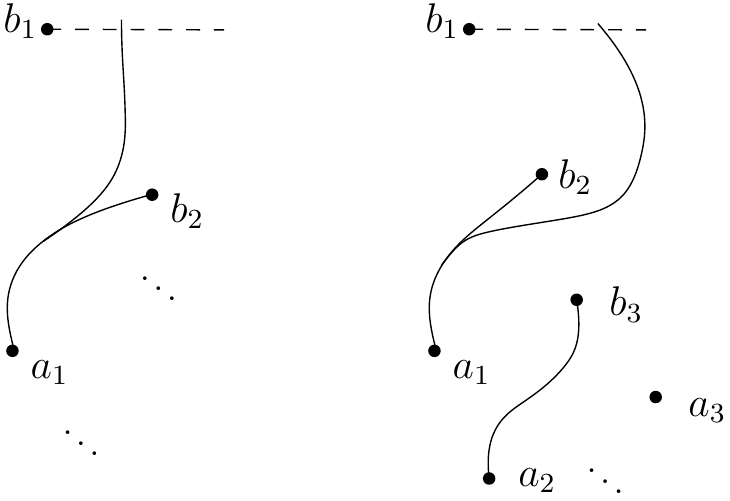}
	\caption{\label{fig:wrong-alternating-cycles}Some situations that are forbidden by Claim~\ref{claim:i-wall}. Left: $b_2$ is `blocked' by a $1$-wall. Right: $a_3$ is blocked by a $2$-wall.}
\end{figure}

\begin{proof}
We will prove the following two implications:
\begin{enumerate}
\item 
If $W$ is an $i$-wall for $C$ with $i \in \{2, \dots, k\}$ such that $b_1$ is to the left of $W$ and $a_i$ to the right of some walk in $W$, then $b_{i+1}$ is also to the right of some walk in $W$. 
(Indices are taken cyclically, as always.)  \\
\item 
If $W$ is an $i$-wall for $C$ with $i \in \{2, \dots, k\}$ such that $b_1$ is to the left of $W$ and $b_i$ is to the right of some walk in $W$, then there is an $i$-wall $W'$ for $C$ with $b_1$ to the left of $W'$ and $a_{i}$ to the right of some walk in $W'$.
\end{enumerate}
Note that these two statements together imply that if at least
one of $a_i, \dots, a_k$, $b_{i}, \dots, b_k$ is to the right of a walk from
an $i$-wall $W$ for $C$ with $i \in \{2, \dots, k\}$ and with $b_1$ to the left of $W$,
then there is a $k$-wall $W'$ for $C$ with $a_k$ to the right of one of its walks and with $b_1$ to the left of $W'$. 
Then applying statement (i) again we get that $b_1$ lies also to the right of some walk of $W'$, which is not possible.
Hence, to establish our observation, it is enough to prove these two statements, which
we do now.

For the proof of the first implication, suppose that we have an $i$-wall $W$ for $C$ with $i \in \{2, \dots, k\}$ and with $b_1$ to the left of $W$ and $a_i$ to the right of some walk in $W$.
Consider a walk $\gamma_i$ witnessing the relation $a_i \leq b_{i+1}$ in $P$.
Note that $\gamma_i$ cannot intersect any walk in $W$.
Indeed, otherwise their intersection would have some element $p$ of $P$ in common (by the planarity of the diagram), and we would have $a_i \leq p \leq b_{i+1}$ in $P$, which together with an extra comparability of the form $a_{\ell} \leq p$ in $P$ or $p \leq b_{\ell+1}$ in $P$ for some $\ell \in \{1, \dots, i-1\}$ contradicts the fact that $C$ is a strict alternating cycle.
Since $a_i$ is to the right of some walk in $W$, by Claim~\ref{claim:topology-of-the-wall} we conclude that the topmost point of $\gamma_i$, namely $b_{i+1}$, is to the right of some walk in $W$, as desired.
(Here we used that $b_{i+1}$ is not above $b_1$ in the drawing;
note that this argument applies even in the special case $i=k$.)

For the proof of the second implication, suppose that we have an $i$-wall $W$ for $C$ with $i \in \set{2, \dots, k}$ with $b_1$ to the left of $W$ and $b_i$ to the right of some walk in $W$.
Recall that $(a_{i},b_{i}) \in I_4$, so there is a walk $\beta$ from $b_{i}$ going downwards to the $a_{i}$-line and intersecting it to the left of $a_{i}$.
Let $W' := W \cup \set{\beta}$.
Since $b_{i}$ is the topmost point of $\beta$, given that $W$ is an $i$-wall for $C$ it should be clear that $W'$ is also an $i$-wall for $C$ and that $a_i$ is to the right of $\beta \in W'$, as desired.
This concludes the proof.
\end{proof}

\subsection{Alternating cycles with special pairs}

Our next goal is to use walls and their properties to show that each strict alternating cycle in $I_4^{\text{no-sep}}$ and each strict alternating cycle in $I_4^{\text{sep}}$ has at least one `special pair':  
A pair $(a_j,b_j)$ of a strict alternating cycle $(a_1,b_1), \dots, (a_k,b_k)$ is said to be {\em special} if 
\begin{enumerate}
\item $a_j$ is drawn below $a_{j+1}$ and $b_{j+1}$ is drawn below $b_j$; \label{item:above}
\item $a_{j+1}$ is to the right of every walk from $a_j$ to $b_{j+1}$, and to the left of every walk from $a_j$ to the $b_j$-line.\label{item:placement-of-aj+1}
\end{enumerate} 
(Indices are taken cyclically.) 
See Figure~\ref{fig:claim-alternating-cycles} for an illustration.

\begin{figure}[h]
\centering
\includegraphics[scale=1.0]{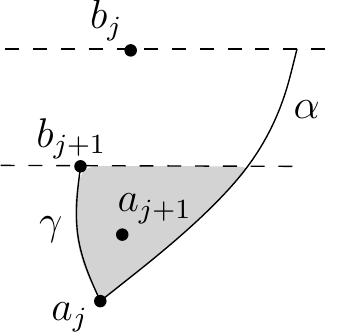}
\caption{\label{fig:claim-alternating-cycles}A special pair $(a_j,b_j)$.}
\end{figure}

First we show that strict alternating cycles in $I_4^{\text{no-sep}}$ have a special pair.

\begin{claim}\label{claim:special-pairs-no-sep}
 Let $C$ denote a strict alternating cycle $(a_1,b_1), \dots, (a_k,b_k)$  in $I_4^{\text{no-sep}}$ with $b_1$ drawn above all other $b_i$'s. 
Then the pair $(a_1,b_1)$ is special.
\end{claim}
\begin{proof}
Arguing by contradiction, suppose that $(a_1,b_1)$ is not special. 
 Since $b_1$ is drawn above $b_2$ and the pair  $(a_1,b_1)$ is not special, we can find a walk $\gamma$ from $a_1$ to $b_2$ witnessing the relation  $a_1 \leq b_2$ in $P$, and a walk $\alpha$ from $a_1$ to the $b_1$-line (hitting that line to the right of $b_1$) such that either
\begin{enumerateAlpha}
\item $a_2$ is to the right of $\alpha$, or \label{enum:right} 
\item $a_2$ is to the left of $\gamma$, or \label{enum:left} 
\item $a_2$ is drawn below $a_1$. \label{enum:below} 
\end{enumerateAlpha} 

Let us start with an easy consequence of Claim~\ref{claim:i-wall}.
Since $\set{\alpha}$ is a $2$-wall for $C$ with $b_1$ to its left, we obtain in particular that $a_2$ and $b_2$ do not lie to the right of $\alpha$.
This already rules out the case~\ref{enum:right}, and it remains to find a contradiction when \ref{enum:left} or \ref{enum:below} hold.

Since $b_2$ is not to the right of $\alpha$ and since $b_2$ is clearly drawn above $a_1$, the bottommost point of $\alpha$, we see that $b_2$ must be to the left of $\alpha$. 
(Note that $b_2$ cannot be on $\alpha$ itself, because $b_2$ is drawn below $b_1$ and $\alpha$ goes up to the $b_1$-line; thus, if $b_2$ were on $\alpha$, then it would follow that $b_2$ is not a maximal element of $P$.)

Now we define a partition of the plane.
Let $D$ be the curve obtained by starting in $b_2$ and going downwards along $\gamma$ until the first intersection point $r$ with $\alpha$, at which point we switch to $\alpha$ and go upwards until its topmost point, which we denote  $q$.
(Note that the intersection of $\alpha$ and $\gamma$ is not empty since $a_1$ belongs to both curves.)
Extend $D$ to the left by adding the horizontal half-line starting at $b_2$ going to the left, and to the right by adding the horizontal half-line starting at $q$ going to the right.
Note that $D$ is not self-intersecting by our previous observation.
Now let $D'$ be the horizontal half-line starting at $r$ going to the right.
The removal of $D\cup D'$ defines three regions of the plane.
We call the region consisting of all points between the $r$-line and the $q$-line that are to the right of $\alpha$ the {\em right} region.
The remaining two regions are referred to as the {\em top} and {\em bottom} regions, in the natural way (see Figure~\ref{fig:3-regions}).

\begin{figure}[h]
\centering
\includegraphics[scale=1.0]{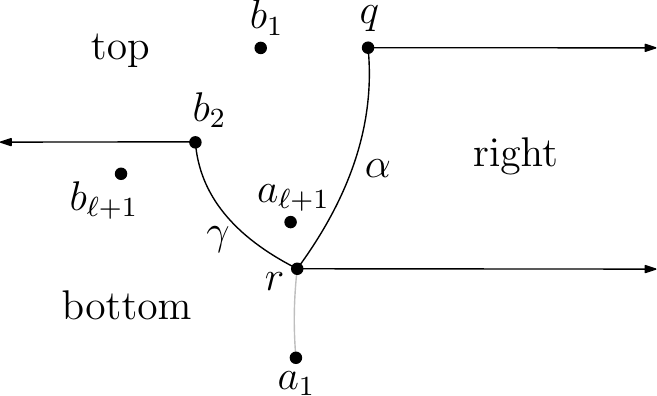}
\caption{\label{fig:3-regions} Illustration of the proof of Claim~\ref{claim:special-pairs-no-sep}. 
	The three regions and the placement of $a_{\ell+1}$, and $b_{\ell+1}$.
}
\end{figure}

Clearly, neither $a_1$ nor $b_1$ is contained in the right region (recall that our regions do not include points from $D\cup D'$).
Together with Claim~\ref{claim:i-wall} applied to the $2$-wall $\set{\alpha}$, we obtain that:
\begin{equation}
\label{eq:no-a-and-no-b-to-the-right-of-alpha}
\parbox{13cm}{
None of $a_1, \dots, a_k, b_1, \dots, b_k$ lies  in the right region.
}
\end{equation}
It follows from~\eqref{eq:no-a-and-no-b-to-the-right-of-alpha} that 
\begin{equation}
\label{eq:a2contained} 
\parbox{13cm}{
$a_2$ is contained in the bottom region
}
\end{equation}
in both cases \ref{enum:left} and \ref{enum:below}.   
We will show that this leads to a contradiction. 
(Whether we are in case \ref{enum:left} or \ref{enum:below} will not be used in the rest of the proof.) 

Let $\ell$ be the largest integer such that $2 \leq\ell\leq k$ and $a_i$ is in the bottom region for every $i\in\set{2,\ldots,\ell}$. 
Such an integer exists by \eqref{eq:a2contained}. 
For each $i\in\set{2,\ldots,\ell}$, let $\beta_i$ be a walk witnessing that $b_i$ sees the left side of $a_i$ (so $a_i$ lies to the right of $\beta_i$).
We claim that:
\begin{equation}
\parbox{0.8\textwidth}{
The set $B_i:=\set{\beta_2,\ldots,\beta_i}$ is an $i$-wall for $C$ with $\beta_2$ being its root walk,
for each $i\in\set{2,\ldots,\ell}$.
Moreover, $b_3,\ldots,b_{\ell+1}$ all lie in the bottom region.
}
\label{eq:Bi-is-an-i-wall}
\end{equation}
First we show that, if $B_i$ ($i\in\set{2,\ldots,\ell}$) is an $i$-wall for $C$ with $\beta_2$ as the root walk, then
$b_{i+1}$ is to the right of some walk in $B_i$ and lies in the bottom region.
To do so we consider a walk $\gamma_{i}$ witnessing the comparability $a_{i} \leq b_{i+1}$ in $P$.
We aim to show now that $\gamma_i$ does not intersect the curve $D$.

The walk $\gamma_{i}$ cannot intersect any walk from $B_{i}$, nor $\gamma$ nor $\alpha$.
Indeed, if  $\gamma_{i}$ did intersect one of these walks, then there would  be an element of $P$ lying in their intersection (by the planarity of the diagram), and this would imply a non-existing comparability in the strict alternating cycle $C$.
Since the bottommost point of $\gamma_i$ (i.e.\ $a_i$) is to the right of $\beta_i$, we deduce from Claim~\ref{claim:topology-of-the-wall} applied to $B_{i}$ that $\gamma_i$ cannot contain a point that is to the left of $b_2$ on the $b_2$-line (recall that $b_2$ is the topmost point of the root $\beta_2$).
Therefore, $\gamma_i$ cannot intersect the horizontal half-line starting at $b_2$ and going left.
Clearly, $\gamma_i$ is disjoint from the horizontal half-line starting at $q$ and going right (since $b_{i+1}$ is not drawn above $b_1$; in fact, $b_{i+1}$ is drawn below $b_1$ unless $i=k$, in which case it is to the left of $q$).
Thus, $\gamma_i$ does not intersect the curve $D$.

Since $a_i$ lies in the bottom region, we conclude that $b_{i+1}$ lies in the bottom or in the right region.
But by property~\eqref{eq:no-a-and-no-b-to-the-right-of-alpha} $b_{i+1}$
cannot be in the right region, and hence it is in the bottom region, as claimed.
Lastly, since $a_i$ is to the right of the walk $\beta_i \in B_i$, using Claim~\ref{claim:topology-of-the-wall} with the walk $\gamma_i$ we deduce that $b_{i+1}$ is also to the right of some walk in $B_i$, as desired.

Now we are ready to prove~\eqref{eq:Bi-is-an-i-wall} by induction on $i$.
The base case $i=2$ is immediate since $B_2=\set{\beta_2}$ is a $2$-wall for $C$.
For the inductive step, assume $i \geq 3$, and let us show that $B_i$ has the desired property.
By the induction hypothesis, $B_{i-1}$ is an $(i-1)$-wall for $C$ with $\beta_2$ as the root walk.
With the observation from the previous paragraph this implies that $b_i$ lies in the bottom region and to the right of some walk in $B_{i-1}$.
This directly yields that $B_{i}=B_{i-1}\cup \set{\beta_{i}}$ is an $i$-wall for $C$ rooted at $\beta_2$.
This concludes the proof of~\eqref{eq:Bi-is-an-i-wall}.

Observe that a corollary of~\eqref{eq:Bi-is-an-i-wall} is that $\ell < k$ (if $\ell=k$ then $b_{k+1}=b_1$ would lie in the bottom region by~\eqref{eq:Bi-is-an-i-wall}, which is clearly not the case).

We are now ready to get a final contradiction for the case under consideration, namely, that all pairs in our alternating cycle are in $I_4^{\text{no-sep}}$.
Recall that, by definition of $I_4^{\text{no-sep}}$, none of $(a_1, b_1), \dots, (a_k, b_k)$ admits a separator.
Using the properties established above, we now exhibit a separator for the pair $(a_{\ell+1}, b_{\ell+1})$, which will be the desired contradiction.

First, note that $b_{\ell+1}$ is in the bottom region (by~\eqref{eq:Bi-is-an-i-wall}) but that $a_{\ell+1}$ is not (by definition of $\ell$).
By~\eqref{eq:no-a-and-no-b-to-the-right-of-alpha}, $a_{\ell+1}$ is not in the right region either.
Since $\ell < k$, we also know that $a_{\ell+1} \neq a_1$, and hence that $a_{\ell+1}$ is not on $D\cup D'$.
This implies that $a_{\ell+1}$ is contained in the top region.
In particular, $a_{\ell+1}$ lies above the $r$-line and consequently so does $b_{\ell+1}$ (as $(a_{\ell+1},b_{\ell+1})\in I_2$).
Since $b_{\ell+1}$ is in the bottom region and is above the $r$-line, $b_{\ell+1}$ must be to the left of $\gamma$.
The element $a_{\ell+1}$, on the other hand, lies to the right of $\gamma$ as this is the only place occupied by the top region below the position of $b_{\ell+1}$.

Finally, every element of $P$ appearing on $\gamma$ is incomparable to both $a_{\ell+1}$ and $b_{\ell+1}$ in $P$, as otherwise this would imply a non-existing comparability in $C$.
This shows that the pair $(a_{\ell+1}, b_{\ell+1})$ is separated by $\gamma$, which is the contradiction we were looking for.
\end{proof}

Let us consider strict alternating cycles in $I_{4}^{\text{sep}}$ now.
We start with a useful observation about $i$-walls with respect to such cycles.
\begin{claim}\label{claim:extend_wall}
Let $C$ denote a strict alternating cycle $(a_1,b_1), \dots, (a_k,b_k)$ in $I_4^{\text{sep}}$ such that $b_1$ is drawn above all other $b_i$'s.
Let $W$ be an $i$-wall for $C$ with $i\in \{2, \dots, k\}$ such that $b_1$ is to the left of $W$ and $a_{i-1}$ belongs to some walk in $W$.
Then for each walk $\lambda$ from $a_i$ going up to some point drawn below $b_1$ such that $\lambda$ intersects the $a_{i-1}$-line and is disjoint from all walks in $W$, we have that the set $W\cup\set{\lambda}$ is an $(i+1)$-wall for $C$ with $b_1$ to its left. 
\end{claim}
\begin{proof}
Let $\lambda$ be an arbitrary walk as in the statement of the claim and let $p$ be the intersection point of $\lambda$ with the $a_{i-1}$-line.
(See Figure~\ref{fig:wall-ext} for an illustration of the current situation and upcoming arguments.)
First, we show that $p$ is to the right of $a_{i-1}$.
Consider a walk $\delta_i$ that separates the pair $(a_i, b_i)$, and let $\gamma_{i-1}$ be a walk witnessing $a_{i-1}\leq b_i$ in $P$.
Observe that both $\lambda$ and $\gamma_{i-1}$ must be disjoint from $\delta_i$, by the definition of a separator.
Since $\gamma_{i-1}$ ends in $b_i$, thus to the left of $\delta_i$, we obtain that the whole walk $\gamma_{i-1}$ stays to the left of $\delta_i$.
Similarly, as $\lambda$ starts in $a_i$, thus to the right of $\delta_i$, we get that $\lambda$ stays to the right of $\delta_i$.
That is to say, $\delta_i$ separates $\gamma_{i-1}$ and $\lambda$ in the diagram.

Now, since all three walks intersect the $a_{i-1}$-line we conclude that $\gamma_{i-1}$ intersects the $a_{i-1}$-line to the left of $\lambda$.
In other words, $p$ lies to the right of $a_{i-1}$ on the $a_{i-1}$-line, and in particular to the right of some walk in $W$ by assumption.
Next, consider the portion $\lambda'$ of $\lambda$ that lies on or above the $p$-line.
Since $\lambda$ is disjoint from walks in $W$ and since the bottommost point of $\lambda'$ (which is $p$) is to the right of some walk in $W$, we deduce by Claim~\ref{claim:topology-of-the-wall} that the topmost point of $\lambda'$ is either to the right of some walk in $W$, or above all walks in $W$.
However, by our assumptions the latter cannot happen as the $b_1$-line is hit by some walk in $W$ but not by $\lambda$.
Therefore, the topmost point of $\lambda$ is to the right of some walk in $W$, and it follows that $W\cup\set{\lambda}$ is an $(i+1)$-wall for $C$ with $b_1$ to its left.  
\end{proof}

\begin{figure}[t]
 \centering
 \includegraphics{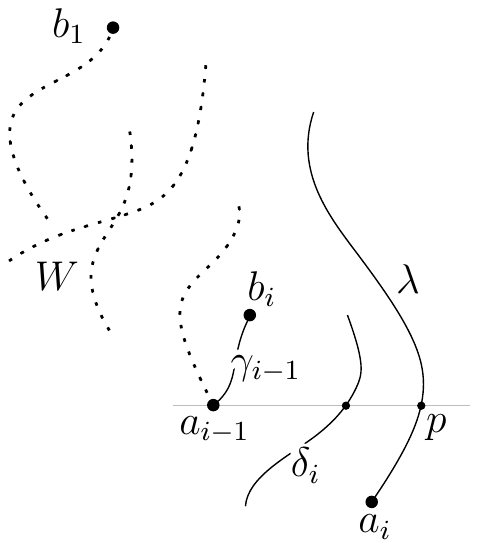}
 \caption{Situation in proof of Claim~\ref{claim:extend_wall}. Dotted lines indicate walks in $W$.}
 \label{fig:wall-ext}
\end{figure}

In the next claim we show that strict alternating cycles in $I_4^{\text{sep}}$ have a special pair.

\begin{claim}\label{claim:special-pairs-sep}
Let $C$ denote a strict alternating cycle $(a_1,b_1), \dots, (a_k,b_k)$  in $I_4^{\text{sep}}$.
Then at least one of the pairs is special. 
\end{claim}

\begin{proof}
We may assume that $b_1$ is drawn above all other $b_i$'s.
Let $j$ be the smallest index in $\set{1, \dots, k}$ such that $a_{j+1}$ is drawn above $a_j$.
(Such an index clearly exists.)
We will show that $(a_j,b_j)$ is a special pair of $C$.

Let $\gamma$ be a walk from $a_j$ to $b_{j+1}$, and let $\alpha$ be a walk from $a_j$ to the $b_j$-line.
We have to show that $b_{j+1}$ is drawn below $b_j$, and that $a_{j+1}$ is to the right of $\gamma$ and to the left of $\alpha$, as illustrated in Figure~\ref{fig:sep-special-pair}
 (left). 

\begin{figure}[h]
 \centering
 \includegraphics{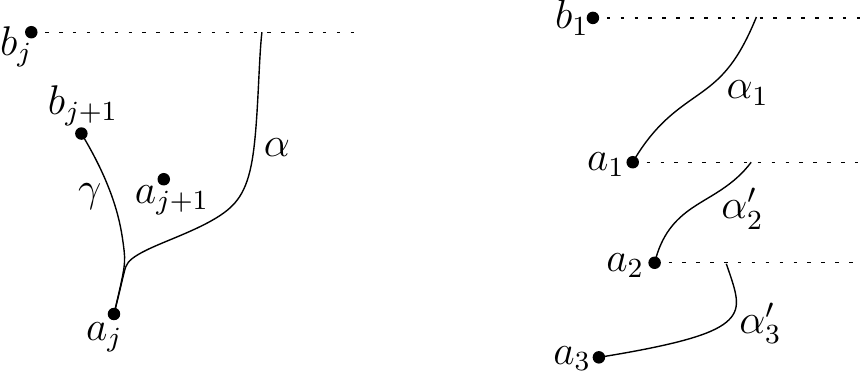}
 \caption{Positions of $a_{j+1}$ and $b_{j+1}$ (left), and the $4$-wall $A_3$ (right).}
 \label{fig:sep-special-pair}
\end{figure}

First let us quickly deal with the $j=1$ case.
In this case, elements $a_1, a_2, b_2, b_1$ appear in this order in the drawing, from bottom to top.
Also, we already know that $a_2$ is to the right of $\gamma$, since $b_2$ only sees the left side of $a_2$.
Furthermore, $\{\alpha\}$ is a $2$-wall for $C$ with $b_1$ to its left. 
Hence, $a_2$ cannot be to the right of $\alpha$ by Claim~\ref{claim:i-wall}, and therefore must be to the left of $\alpha$.
This proves the statement for $j=1$.

Next, assume $j \geq 2$.
We set up some walls as follows.
For each $i\in \{1, \dots, j\}$, let $\alpha_i$ denote a walk from $a_i$ to the $b_i$-line.
Also, if $i \geq 2$, let $\alpha'_i$ denote the portion of $\alpha_i$ that goes from $a_i$ to the $a_{i-1}$-line.
Define $A_i:=\{\alpha_1\} \cup \{\alpha'_2, \dots, \alpha'_i\}$, for each $i\in \{1, \dots, j\}$.
See Figure~\ref{fig:sep-special-pair} (right) for an illustration of how the walks in $A_3$ might look like.

We claim that $A_i$ is an $(i+1)$-wall for $C$ with $b_1$ to the left of $A_i$, which we prove by induction on $i$.
For the base case $i=1$, it is clear that $A_1=\{\alpha_1\}$ is a $2$-wall for $C$ with $b_1$ to its left (since $\alpha_1$ hits the $b_1$-line to the right of $b_1$).
Now assume $i\geq 2$ for the inductive step.
By induction, $A_{i-1}=\{\alpha_1, \alpha'_2, \dots, \alpha'_{i-1}\}$ is an $i$-wall for $C$ with $b_1$ to the left of $A_{i-1}$.
Applying Claim~\ref{claim:extend_wall} to the $i$-wall $A_{i-1}$ and the walk $\alpha'_i$, which is clearly disjoint from all walks in $A_{i-1}$, we
obtain that $A_{i-1} \cup \{\alpha'_i\} = A_i$ is an $(i+1)$-wall for $C$ with $b_1$ to its left, as desired.

Let us point out the following consequence of Claim~\ref{claim:extend_wall} when applied to the $j$-wall $A_{j-1}$:
\begin{equation}
\label{eq:using_separator}
\parbox{0.8\textwidth}{
If $\lambda$ is a walk from $a_j$ to some point drawn below $b_1$ such that $\lambda$ hits the $a_{j-1}$-line and is disjoint from all walks in $A_{j-1}$, then $a_{j+1}$
cannot be to the right of $\lambda$. }
\end{equation}

Indeed, if $a_{j+1}$ is to the right of $\lambda$, then $a_{j+1}$ is to the right of a walk from the
set $A_{j-1} \cup \{\lambda\}$, which is a $(j+1)$-wall with $b_1$ to its left by Claim~\ref{claim:extend_wall}.
However, this is not possible by Claim~\ref{claim:i-wall}.

We may now conclude the proof.
Recall that $a_{j+1}$ is drawn above $a_j$.
We already know that $a_{j+1}$ is to the right of $\gamma$, since $b_{j+1}$ only sees the left side of $a_{j+1}$.
We need to show that $b_{j+1}$ is drawn below $b_j$, and that $a_{j+1}$ lies to the left of $\alpha$.

Note that $\gamma$ is disjoint from all walks in the $j$-wall $A_{j-1}$ as each intersection would imply a non-existing comparability within the strict alternating cycle $C$.
Thus if $\gamma$ intersects the $a_{j-1}$-line, then property~\eqref{eq:using_separator} applies to $\gamma$ and yields that $a_{j+1}$ is to the left of $\gamma$, which is not true. 
Thus, $\gamma$ does not intersect the $a_{j-1}$-line.
In particular, $b_{j+1}$ is drawn below $a_{j-1}$, which is itself drawn below $b_j$.
Hence, $b_{j+1}$ is drawn below $b_j$.

It only remains to show that $a_{j+1}$ is to the left of $\alpha$.
Let $\alpha'$ be the portion of $\alpha$ ranging from $a_j$ to the $a_{j-1}$-line.
Recall that $a_{j+1}$ is drawn above $a_j$ and below $b_{j+1}$ (which itself is drawn below $a_{j-1}$).
Since $a_{j+1}$ cannot be on $\alpha'$ (because $C$ is a strict alternating cycle), we deduce that $a_{j+1}$ is either to the left of $\alpha'$ or to the right of $\alpha'$.
However, property~\eqref{eq:using_separator} applies to $\alpha'$ and hence $a_{j+1}$ must lie to the left of $\alpha'$.
Thus in particular, $a_{j+1}$ is to the left of $\alpha$.
This completes the proof of the claim.
\end{proof}

\subsection{Auxiliary directed graph}
 
We are now ready to go back to our main objective, which is to show that the following three sets can be partitioned into at most $h$ reversible sets each: Non-dangerous pairs in $I_4^{\text{sep}}$, non-dangerous pairs in $I_4^{\text{no-sep}}$, and dangerous pairs in $I_4^{\text{no-sep}}$.  
(Recall that an incomparable pair $(a,b) \in I_4$ is said to be dangerous if $a$ is drawn below $x_0$ and $a$ sees the left side of $x_0$.)  
To do so, we define an auxiliary directed graph on each of these three sets:  
Suppose that $J$ denotes one of these three sets.   
Define a directed graph $G$ with vertex set $J$ as follows.     
Given two distinct pairs $(a,b),(a',b')\in J$, we put an edge from $(a,b)$ to $(a',b')$ in $G$ if there exists a strict alternating cycle $(a_1, b_1), \dots, (a_k,b_k)$ in $J$ satisfying the following three properties:
\begin{enumerate}
\item $b_1$ is drawn topmost among all $b_i$'s; 
\item at least one of the pairs is special, and 
\item $(a,b)=(a_j,b_j)$ and $(a',b')=(a_{j+1},b_{j+1})$, where $j$ is the smallest index in $\set{1,\ldots,k}$ such that $(a_j,b_j)$ is special. 
\end{enumerate}   
In this case, we say that the alternating cycle is a \emph{witness} for the edge $((a,b), (a',b'))$.  
Observe that when there is an edge from $(a,b)$ to $(a',b')$ in $G$ then $a$ is drawn below $a'$. 
This implies in particular that $G$ has no directed cycle. 
Define the {\em chromatic number} $\chi(G)$ of $G$ as the chromatic number of its underlying undirected graph. 
Since every strict alternating cycle in $J$ has at least one special pair (c.f.\ Claims~\ref{claim:special-pairs-no-sep} and~\ref{claim:special-pairs-sep}), we have 
$$\dim(J) \leq \chi(G).$$  
Thus it is enough to show that $\chi(G) \leq h$. 
We will do so by showing that every directed path in $G$ has at most $h$ vertices. 
This is done in the next two sections; the proofs will be different depending on whether the pairs in $J$ are dangerous or not. 
In the meantime, let us make some observations on directed paths in $G$ that hold independently of the type of pairs in $J$. 

Let $(a^1, b^1), \dots, (a^{\ell}, b^{\ell})$ denote any directed path in $G$. 
(To avoid confusion, we use superscripts when considering directed paths and subscripts when considering alternating cycles.) 
Note that the $a^i$'s go up in the diagram, while the $b^i$'s go down: Indeed, for each $i\in  \{1,\dots, \ell-1\}$ the two pairs $(a^i, b^i)$ and $(a^{i+1}, b^{i+1})$ appear consecutively in that order in some strict alternating cycle in $J$ where $(a^i, b^i)$ is special, and thus $a^{i+1}$ is drawn above $a^i$ and $b^{i+1}$ is drawn below $b^i$.  

Now, for each $i \in \{1, \dots, \ell-1\}$, let $\gamma^i$ denote an arbitrary walk from $a^i$ to $b^{i+1}$ witnessing the relation $a^i \leq b^{i+1}$ in $P$. 
The following observation will be useful in our proofs: 
\begin{equation}
\label{eq:completely_to_the_right}
\textrm{
For each $i\in \{1, \dots, \ell-2\}$, the walk $\gamma^{i+1}$ is completely to the right of $\gamma^i$.
}
\end{equation}

\begin{figure}[t]
 \centering
 \includegraphics{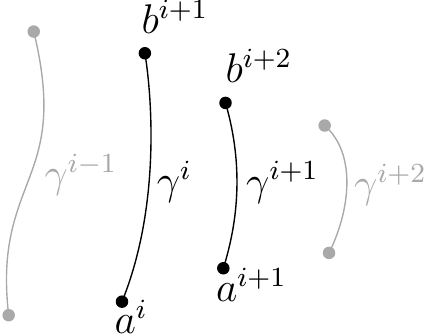}
 \caption{How the $\gamma^i$'s look like.}
 \label{fig:gamma-walks}
\end{figure}

By `completely to the right', we mean that every point of $\gamma^{i+1}$ is to the right of $\gamma^i$ (as in Figure~\ref{fig:gamma-walks}).
To see this, recall once again that $a^i$ is drawn below $a^{i+1}$ and that $b^{i+2}$ is drawn below $b^{i+1}$. 
Moreover, $\gamma^i$ and $\gamma^{i+1}$ cannot intersect because $a^{i+1}$ and $b^{i+1}$ are incomparable in $P$. 
Finally, $a^{i+1}$ is to the right of $\gamma^i$, as follows from the fact that $(a^i, b^i)$ is a special pair of some alternating cycle in $J$ where $(a^i,b^i), (a^{i+1}, b^{i+1})$ appear consecutively in that order. 
Altogether, this implies that $\gamma^{i+1}$ is completely to the right of $\gamma^i$. 

We continue our study of directed paths in $G$ in the next two sections. 
 
\subsection{Non-dangerous pairs} 

In this section we consider pairs in $I_4^{\text{sep}}$ and in $I_4^{\text{no-sep}}$ that are not dangerous. 
Our approach is independent of whether the pairs have separators or not, thus in this section we fix $J$ as the subset of non-dangerous pairs in one of these two sets. 
We show that $\chi(G) \leq h$, where $G$ denotes the directed graph on $J$ defined earlier: 

\begin{claim}
Every directed path in $G$ has at most $h$ vertices, and hence $\dim(J) \leq \chi(G) \leq h$.
\end{claim}
\begin{proof}
Let $(a^1, b^1), \dots, (a^{\ell}, b^{\ell})$ denote any directed path in $G$. 
For $i\in\set{1,\ldots,\ell-1}$, let $\gamma^i$ be a walk witnessing the relation $a^i\leq b^{i+1}$ in $P$ that is `leftmost' in the diagram among all such walks: 
For every point $p$ of $\gamma^i$ and every walk $\beta$ witnessing that relation, either $p$ is also on $\beta$, or $p$ is to the left of $\beta$. 
(A little thought shows that $\gamma^i$ is well defined, and uniquely defined.) 

Let $q$ be the largest index in $\{1, \dots, \ell-1\}$ such that $a^q$ is drawn below $x_0$ if there is such an index, otherwise set $q:=0$. 
To illustrate the usefulness of property~\eqref{eq:completely_to_the_right}, we use it to show that $h \geq q+1$.  
Assuming $q > 0$ (otherwise the claim is vacuous), we first note that  $x_0$ is to the left of $\gamma^1$, since $a^1$ is drawn below $x_0$ and sees only its right side. (This is where we use that $(a^1, b^1)$ is not dangerous.) 
Fix some walk $\alpha$ witnessing the relation $x_0 \leq b^{\ell}$ in $P$. 
Since $a^i$ is below the $x_0$-line for each $i=1, \dots, q$, we directly deduce from~\eqref{eq:completely_to_the_right} that $\alpha$ must intersect each of $\gamma^1, \dots, \gamma^{q}$; see Figure~\ref{fig:piercing-gammas}.  
Given that these walks are disjoint, and that each intersection contains at least one element from $P$, counting $x_0$ we conclude that the height of $P$ is at least $q+1$.  
(In fact, we almost get a lower bound of $q+2$ by counting $b_{\ell}$; however, note that $\alpha$ could intersect $\gamma^{q}$ on $b_{q+1}$, which coincides with $b_{\ell}$ when $q=\ell-1$.)  

\begin{figure}[t]
 \centering
 \includegraphics{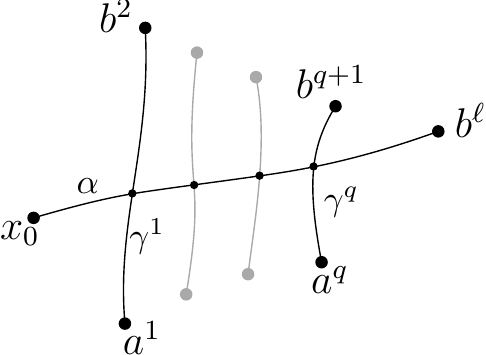}
 \caption{Each walk from $x_0$ to $b^\ell$ intersects the $q$ walks $\gamma^1, \dots, \gamma^q$.}
 \label{fig:piercing-gammas}
\end{figure}

Thus we are already done with the proof if $q=\ell-1$. 
Assume from now on that $q < \ell-1$ (recall that possibly $q=0$).  
Ideally, we would like to argue that $\alpha$ must intersect all of $\gamma^{q+1},\dots, \gamma^{\ell-1}$ as well, which would be enough to conclude the proof.  
However, this is not necessarily true: For instance, $\alpha$ could avoid $\gamma^{q+1}$ by passing under $a^{q+1}$ and then going to its right. 
In the rest of the proof, we will enrich our current structure---i.e.\ the walks $\gamma^1, \dots, \gamma^{\ell-1}$---in a way that will allow us to conclude that $\alpha$ intersects at least $\ell-1-q$ disjoint walks, that are moreover disjoint from $\gamma^1, \dots, \gamma^q$.  
This will show that $\alpha$ contains at least $q + 1 + (\ell-1-q)=\ell$ elements of $P$, and hence that $h \geq \ell$, as desired.

\begin{figure}[t]
 \centering
 \includegraphics{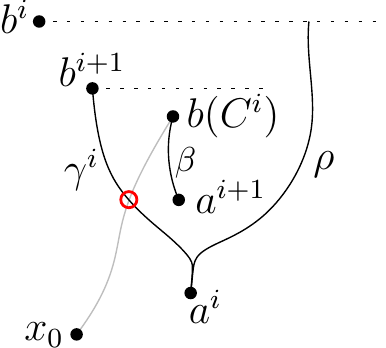}
 \caption{Situation if we assume that $b(C^i)$ is drawn below $b^{i+1}$. Considering any walk from $x_0$ to $b(C^i)$, we see a forbidden comparability in the poset.}
 \label{fig:cycle-properties}
\end{figure}

For each $i\in\set{q+1,\ldots,\ell-1}$, choose some strict alternating cycle $C^i$ in $J$ witnessing the edge $((a^i,b^i),(a^{i+1},b^{i+1}))$. 
Thus $(a^i,b^i), (a^{i+1},b^{i+1})$ appear consecutively in that order in $C^i$, and $(a^i,b^i)$ is a special pair of $C^i$.
The element $b$ of the pair $(a,b)$ appearing just after $(a^{i+1},b^{i+1})$ in $C^i$ will be important for our purposes, let us denote it by $b(C^i)$. 
(Note that $b(C^i) = b^i$ in case $C^i$ is of length $2$.)
We show: 
\begin{equation}
\label{eq:bCabove}
\textrm{
$b(C^i)$ is drawn above $b^{i+1}$ for each $i\in\set{q+1,\ldots,\ell-1}$. 
}
\end{equation}

To prove this, let $i\in\set{q+1,\ldots,\ell-1}$ and consider the alternating cycle $C^i$. 
First, if $C^i$ has length $2$, then  $b(C^i) = b^i$, which is indeed drawn above $b^{i+1}$. 
Next, assume that $C^i$ has length at least $3$.  
Arguing by contradiction, suppose that $b(C^i)$ is drawn below $b^{i+1}$. 
Let $\rho$ denote a walk starting in $a^i$ and hitting the $b^i$-line to the right of $b^i$.
 Since $(a^i,b^i)$ is a special pair of $C^i$, we know that $a^{i+1}$ lies to the right of $\gamma^i$ and to the left of $\rho$.
 We claim that $b(C^i)$ also lies to the right of $\gamma^i$ and to the left of $\rho$, as depicted in Figure~\ref{fig:cycle-properties}.
 Since $b(C^i)$ is drawn below $b^{i+1}$ and $b^i$, both walks $\gamma^i$ and $\rho$ hit the $b(C^i)$-line.
  Let $\beta$ denote a walk witnessing the comparability $a^{i+1}\leq b(C^i)$ in $P$. 
Now observe that both $\gamma^i$ and $\rho$ are disjoint from $\beta$ as otherwise this would imply $a^i\leq b(C^i)$ in $P$, contradicting the fact that the alternating cycle $C^i$ is strict (here we use that $C^i$ has length at least $3$). 
Given that $a^{i+1}$ is to the right of $\gamma^i$ and to the left of $\rho$, we deduce from this discussion that $\beta$ is completely to the right of $\gamma^i$, and completely to the left of $\rho$. 
 In particular, $b(C^i)$ is to the right of $\gamma^i$ and to the left of $\rho$, as claimed. 
 Since $x_0$ is drawn below $a^i$, this implies that any walk witnessing the comparability $x_0 \leq b(C^i)$ in $P$ must intersect at least one of $\gamma^i$ and $\rho$ (by planarity of the diagram). 
 It follows that $a^i \leq b(C^i)$ in $P$, contradicting the assumption that $C^i$ is strict.
This concludes the proof of~\eqref{eq:bCabove}.

We define the following additional walks:  
Let  $\delta^{q+1}$ denote a walk from $a^{q+1}$ to the $b^{q+1}$-line, and for each $i\in\set{q+2,\ldots,\ell}$, let $\delta^i$ denote a walk from $a^{i}$ to $b(C^{i-1})$ witnessing the relation $a^{i}\leq b(C^{i-1})$ in $P$. 
(We remark that $\delta^{q+1}$ is defined slightly differently from  $\delta^{q+2}, \dots,\delta^{\ell}$; fortunately, this will not cause any complication in the arguments below.) 
In the remaining part of the proof, the key step will be to show that the different walks create together a `nested structure' around $b^{\ell}$ with $x_0$ completely outside of it, as illustrated in Figure~\ref{fig:nested-struct}.   
Using this nested structure, we will easily conclude that the walk $\alpha$ intersects $\ell-1-q$ pairwise disjoint walks, each disjoint from $\gamma^1, \dots, \gamma^q$ and from $x_0$, as desired. 

\begin{figure}[t]
 \centering
 \includegraphics{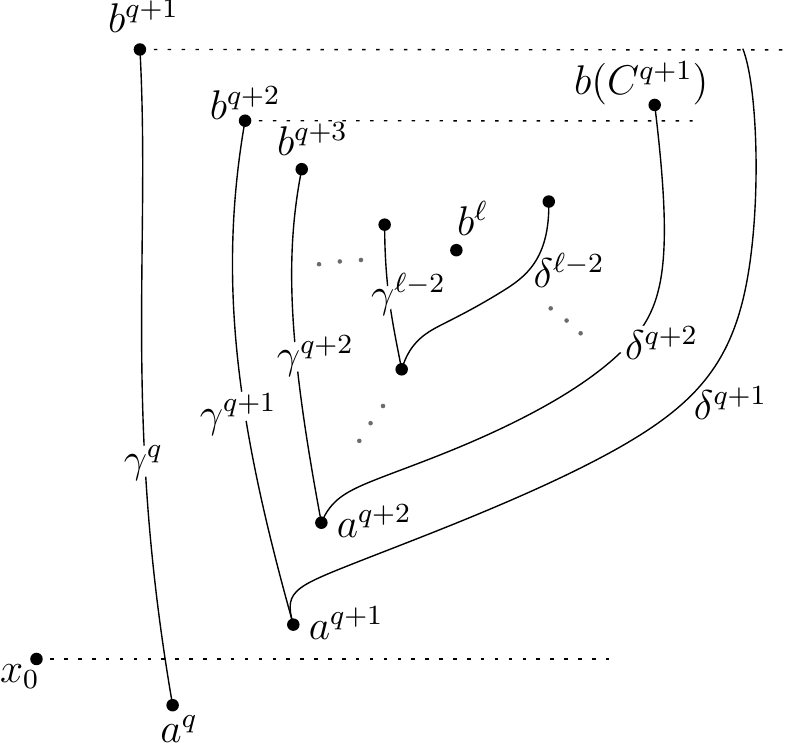}
 \caption{Nested structure built by the walks $\gamma^{q+1}, \dots, \gamma^{\ell-2}$ and $\delta^{q+1}, \dots, \delta^{\ell-2}$.}
 \label{fig:nested-struct}
\end{figure}
 
For each $i\in\set{q+1,\ldots,\ell}$, let $\tilde \delta^{i}$ denote the portion of $\delta^{i}$ from $a^{i}$ to the $b^{i}$-line. 
Thus $\tilde \delta^{q+1} = \delta^{q+1}$, and for $i \geq q+2$ the walk $\tilde \delta^{i}$ is a proper prefix of $\delta^{i}$ since $b(C^{i-1})$ is drawn above $b^{i}$ by~\eqref{eq:bCabove}. 

Let us point out some properties of the pair of walks $\gamma^i$ and $\tilde \delta^{i+1}$ for $i\in\set{q+1,\ldots,\ell-1}$:   
First, $\gamma^i$ and $\tilde \delta^{i+1}$ are disjoint as otherwise we would deduce that $a^{i+1}\leq b^{i+1}$ in $P$.
Recall also that $a^{i+1}$ lies to the right of $\gamma^i$, implying that $\tilde \delta^{i+1}$ starts to the right of $\gamma^i$.
By the disjointness of the two walks, this in turn implies that $\tilde \delta^{i+1}$ is completely to the right of $\gamma^i$ (here we use that both walks have their top endpoints on the $b^{i+1}$-line). 

 Building on the observations above, we show: 
 \begin{enumerate}
  \item $\tilde \delta^{i+1}$ is completely to the left of $\tilde \delta^{i}$, for each $i\in\set{q+1,\ldots,\ell-1}$,\label{item:passing-right} \\[.1ex] 
  \item $\gamma^i$ and $\tilde \delta^j$ are disjoint, for all $i\in\set{1,\ldots,\ell-1}$ and $j\in\set{q+1,\ldots,\ell}$ such that $i\neq j$.
  \label{item:disjoint-gamma-delta}
 \end{enumerate}

 Let us first prove~\ref{item:passing-right}. 
 Since $\tilde \delta^i$ starts in $a^{i}$ and ends on the $b^{i}$-line, 
 we see that $\tilde \delta^i$ starts below the whole walk $\tilde \delta^{i+1}$ and ends above $\tilde \delta^{i+1}$. 
Using that $(a^{i},b^{i})$ is a special pair of $C^{i}$, we also see that $a^{i+1}$ is to the left of $\delta^{i}$. 
Thus, to prove~\ref{item:passing-right} it is enough to show that $\delta^i$ and $\delta^{i+1}$ are disjoint. 
Recall that  $\delta^{i+1}$ witnesses the comparability $a^{i+1}\leq b(C^{i})$ in $P$, since $i+1 \geq q+2$. 
If $\delta^i$ and $\delta^{i+1}$ intersect, then we deduce that $a^{i}\leq b(C^{i})$ in $P$, in contradiction with the fact that $C^i$ is a strict alternating cycle.  
So $\delta^i$ and $\delta^{i+1}$  must be disjoint.

 Next we show~\ref{item:disjoint-gamma-delta}. 
First we consider the case $j \leq i-1$. 
By~\ref{item:passing-right}, $\tilde \delta^{i}$ is completely to the left of $\tilde \delta^{i-1}$. 
Iterating, we see that  $\tilde \delta^{i}$ is completely to the left of $\tilde \delta^{j}$. 
On the other hand, every point $p$ in  $\tilde \delta^{i}$ is either also on $\gamma^i$, or is to the right of $\gamma^i$. 
This is because of the `leftmost' property of $\gamma^i$: The walk $\tilde \delta^{i}$ starts in $a^i$ and hits the $b^{i+1}$-line to the right of $b^{i+1}$; if some point of  $\tilde \delta^{i}$ were to the left of  $\gamma^i$ then there would be a way to `re-route' a part of $\gamma^i$ to the left using the appropriate part of $\tilde \delta^{i}$, contradicting the fact that $\gamma^i$ has been chosen as the leftmost walk from $a^i$ to $b^{i+1}$.   
We deduce that  $\gamma^i$ is completely to the left of $\tilde \delta^{j}$, and thus in particular that the two walks are disjoint.  

Next we consider the case $j\geq i+1$.    
Here we use the fact that $\gamma^j$ is completely to the right of $\gamma^i$ (c.f.~\eqref{eq:completely_to_the_right}):  
 Since every point of $\gamma^j$ is either on $\tilde \delta^j$ or to the left of it (by the leftmost property of $\gamma^j$, exactly as in the previous paragraph), we conclude that  $\tilde \delta^j$ is also completely to the right of $\gamma^i$. 
In particular, they are disjoint. 
 This completes the proof of \ref{item:disjoint-gamma-delta}. \\

 We are finally ready to show that $\ell\leq h$.
 Recall that we proved at the beginning of the proof that the walk $\alpha$ intersects each of $\gamma^1, \dots, \gamma^q$. 
Counting $x_0$, these intersections already single out $q+1$ elements of $P$ on $\alpha$. 
Using the walks we defined above, we now identify $\ell -q -1$ extra elements of $P$ on $\alpha$, implying that $P$ has height at least $\ell$.  

 Let $i\in\set{q+1,\ldots,\ell-2}$.
 The walks $\gamma^i$ and $\tilde \delta^{i}$ both start in $a^i$.
 Recall also that $\gamma^i$ hits the $b^\ell$-line to the left of $b^\ell$.
 Since $\tilde \delta^{\ell}$ hits the $b^\ell$-line to the right of $b^\ell$, we deduce from~\ref{item:passing-right} that the same is true for $\tilde \delta^{i}$.
 Recall also that $x_0$ is drawn below $a^i$, since $i>q$. 
 From these observations it follows that the walk $\alpha$ has to intersect the union $U_i$ of the two walks $\gamma^i$ and $\delta^{i}$  in an element of $P$.
 Using~\ref{item:disjoint-gamma-delta} we see that $U_i$ is disjoint from $U_j$ for all $j\in\set{q+1,\ldots,\ell-2}$ with $j\neq i$, and from $\gamma^1, \dots, \gamma^q$ as well. 

Hence considering the intersection of $\alpha$ with $U_i$ for $i=q+1, \dots, \ell-2$ we identify $\ell-q-2$ `new' elements of $P$ on $\alpha$. 
Finally, we can get an extra one by observing that element $b^{\ell}$ (the top endpoint of $\alpha$) has not been counted so far.  
This concludes the proof. 
\end{proof} 

\subsection{Dangerous pairs}

It only remains to partition the set $J$ of dangerous pairs in $I_{4}^{\text{no-sep}}$ into reversible sets. 
As in the previous section, let us show that $\chi(G) \leq h$ for the directed graph $G$ we defined on $J$: 

\begin{claim}
Every directed path in $G$ has at most $h$ vertices, and hence $\dim(J) \leq \chi(G) \leq h$.
\end{claim}
\begin{proof}
Let $(a^1, b^1), \dots, (a^{\ell}, b^{\ell})$ denote any directed path in $G$.  
For each $i\in\set{1,\ldots,\ell-1}$, let $C^i$ denote a strict alternating cycle in $J$ witnessing the edge $((a^i,b^i),(a^{i+1},b^{i+1}))$. 
Note that, thanks to Claim~\ref{claim:special-pairs-no-sep}, we know that $b^i$ is drawn topmost among all elements of $P$ appearing in $C^i$.   
In this proof, the element $a$ of the pair $(a,b)$ appearing just before $(a^{i},b^{i})$ in $C^i$ will play a special role, let us denote it by $a(C^i)$. 
(Observe that $a(C^i) = a^{i+1}$ in case $C^i$ is of length $2$.) 
Let $\delta^i$ denote a walk witnessing the relation $a(C^i) \leq b^i$ in $P$.

A key property of the walks $\delta^1, \dots, \delta^{\ell-1}$ is the following:  
\begin{equation}
\label{eq:aC}
\textrm{
$x_0$ and $b^{i+1}$ are both to the left of $\delta^i$ for each $i\in\set{1,\ldots,\ell-1}$. 
}
\end{equation}
Figure~\ref{fig:gamma-vs-delta} illustrates this situation.
To show this, let $i\in\set{1,\ldots,\ell-1}$ and let $(c_1, d_1), \dots, (c_k, d_k)$ denote the pairs forming the cycle $C^i$ in order, with $(c_1, d_1) = (a^i, b^i)$,  $(c_2, d_2) = (a^{i+1}, b^{i+1})$, and $c_k = a(C^i)$.  

\begin{figure}[t]
 \centering
 \includegraphics[scale=1.0]{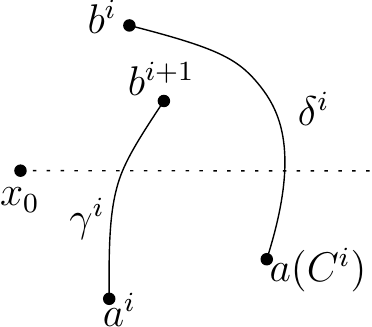}
 \caption{$x_0$ and $b^{i+1}$ are to the left of $\delta^i$.}
 \label{fig:gamma-vs-delta}
\end{figure}

Recall that for each $j\in\set{2,\ldots,k}$ the element $d_j$ is drawn above $c_k$ (since the $x_0$-line lies in between $c_k$ and $d_j$) and below $d_1$ (by our choice of $C^i$). 
In particular, $d_j$ is either to the left or to the right of $\delta^i$. 
(Note that $d_j$ cannot be on $\delta^i$ since $j\neq 1$.) 
We will prove that all of $d_2, \dots, d_k$ are to the left of $\delta^i$, and thus in particular $d_2 = b^{i+1}$ is. 

We start by showing that $d_k$ is to the left of $\delta^i$. 
Suppose not. 
Let $\beta$ be a walk witnessing the fact that $d_k$ sees the left side of $c_k$. 
Since $\beta$'s bottom endpoint is to the left of $\delta^i$ and its top endpoint $d_k$ is to the right of $\delta^i$, it follows that these two walks intersect (by planarity of the diagram). 
However, this implies $c_k \leq d_k$ in $P$, a contradiction. 
Hence $d_k$ is to the left of $\delta^i$ as claimed. 

Next, we observe that $x_0$ and $d_j$ are on the same side of $\delta^i$ for each $j\in\set{2,\ldots,k}$. 
For if $x_0$ and $d_j$ are on different sides of $\delta^i$, then any walk witnessing the relation $x_0 \leq d_j$ in $P$ intersects $\delta^i$, which in turn implies that $c_k \leq d_j$ in $P$, in contradiction with the fact that $C^i$ is a strict alternating cycle.  
(Notice that the fact that $d_1=b^i$ is drawn above $d_j$ is crucial here.)  
Combining this with the fact that $d_k$ is to the left of $\delta^i$, we deduce that $x_0$ and $d_2, \dots, d_k$ are all to the left of $\delta^i$, which establishes~\eqref{eq:aC}. 

For $i\in\set{1,\ldots,\ell-1}$, let $\gamma^i$ be an arbitrary walk witnessing the relation $a^i\leq b^{i+1}$ in $P$, let $\tilde \delta^i$ denote the portion of $\delta^i$ extending from $b^i$ to the $x_0$-line, and let similarly $\tilde \gamma^i$ denote the portion of $\gamma^i$ from $b^{i+1}$ to the $x_0$-line.     
Using~\eqref{eq:aC}, we aim to show that these walks build a `nested structure' as depicted in Figure~\ref{fig:nesting-top}.
To do so, we need to prove: 
 \begin{enumerate}
  \item $\tilde \delta^{i+1}$ is completely to the left of $\tilde \delta^{i}$, for each $i\in\set{1,\ldots,\ell-2}$,\label{item:passing-right_dangerous} \\[.1ex] 
  \item $\tilde \gamma^i$ is completely to the left of $\tilde \delta^{i}$, for each $i\in\set{1,\ldots,\ell-1}$, \label{item:disjoint-gamma-delta_dangerous} \\[.1ex] 
\item $\gamma^1, \dots, \gamma^{\ell-1}, \delta^{\ell-1}, \dots, \delta^1$ cross the $x_0$-line in this order from left to right,  \label{item:order_dangerous} \\[.1ex] 
\item $x_0$ is to the right of $\gamma^{\ell-2}$ and to the left of $\delta^{\ell-1}$. 
\label{item:x0_dangerous} 
 \end{enumerate}

In order to motivate these four properties, let us use them to show that $h \geq \ell$. 
For $i\in\set{2,\ldots,\ell-1}$, let $U_i$ denote the union of $\tilde \gamma^{i-1}$ and $\tilde \delta^i$, both of which have $b^i$ as top endpoint.  
Recalling that $\gamma^{i}$ is completely to the right of $\gamma^{i-1}$ (c.f.~\eqref{eq:completely_to_the_right}), it follows from~\ref{item:passing-right_dangerous}, \ref{item:disjoint-gamma-delta_dangerous}, and \ref{item:order_dangerous} that $U_2, \dots, U_{\ell-1}$ are pairwise disjoint, and moreover that they form a nested structure: 
Say that a point $p$ is {\em inside} $U_i$ if $p$ is to the right of $\tilde \gamma^{i-1}$ and to the left of $\delta^i$. 
Then every point inside $U_i$ is also inside $U_{i-1}, \dots, U_2$.
Since $x_0$ is inside $U_{\ell-1}$ by~\ref{item:x0_dangerous}, this implies that any walk witnessing the relation $x_0 \leq b^1$ in $P$ intersects each of $U_{2}, \dots, U_{\ell-1}$; see Figure~\ref{fig:nesting-top} illustrating this situation. 
Counting $x_0$ and $b^1$, this shows that $P$ has height at least $\ell$, as desired. 
  
\begin{figure}[t]
 \centering
 \includegraphics[scale=1.0]{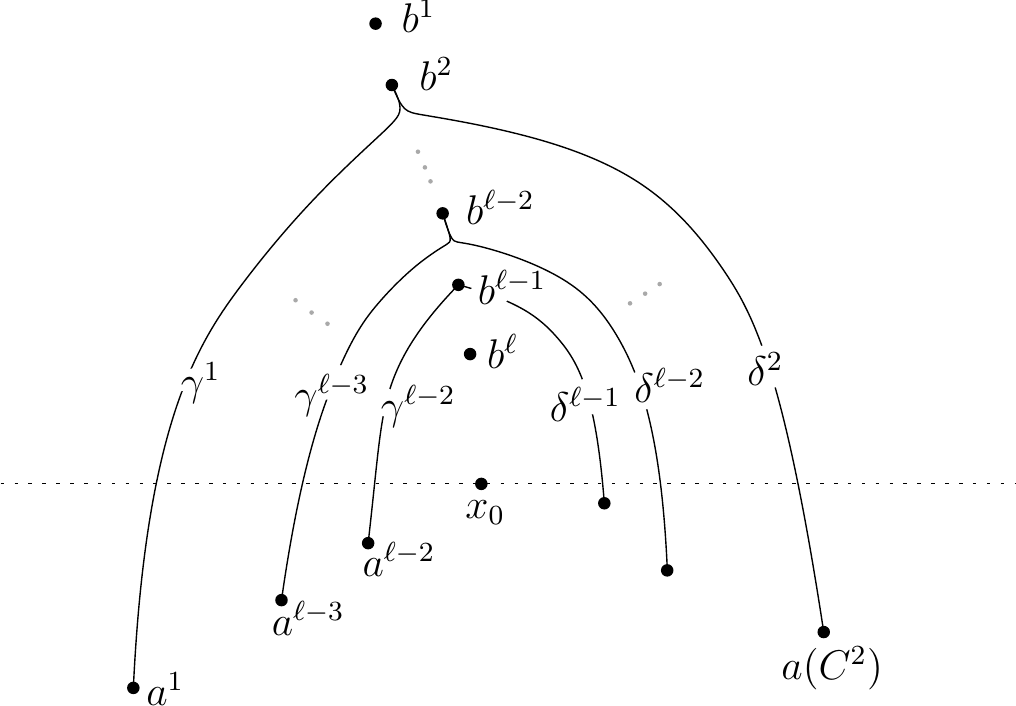}
 \caption{Nested structure formed by the walks $\gamma^1,\ldots,\gamma^{\ell-2}$ and $\delta^2,\ldots,\delta^{\ell-1}$.}
 \label{fig:nesting-top}
\end{figure}

Thus it only remains to prove the four properties above. 
To show~\ref{item:passing-right_dangerous}, first note that $\delta^i$ and $\delta^{i+1}$ must be disjoint, for otherwise we would have $a(C^i) \leq b^{i+1}$ in $P$, contradicting the fact that $C^i$ is a strict alternating cycle.  
Thus, using that $b^{i+1}$ is to the left of $\tilde \delta^{i}$ by~\eqref{eq:aC}, we deduce that  $\tilde \delta^{i+1}$ is completely to the left of $\tilde \delta^{i}$. 

Property~\ref{item:disjoint-gamma-delta_dangerous} is a consequence of~\eqref{eq:aC}: 
$b^{i+1}$ is to the left of $\delta^i$ by~\eqref{eq:aC}, and $\gamma^i$ and $\delta^i$ cannot intersect as this would imply $a^i \leq b^i$ in $P$. 

Property~\ref{item:order_dangerous} follows directly from~\eqref{eq:completely_to_the_right}, \ref{item:passing-right_dangerous}, and~\ref{item:disjoint-gamma-delta_dangerous}. 

Finally, to establish~\ref{item:x0_dangerous}, we only need to show that $x_0$ is to the right of $\gamma^{\ell-2}$ since we already know that $x_0$ is to the left of $\delta^{\ell-1}$ by~\eqref{eq:aC}.  
Arguing by contradiction, suppose that $x_0$ is to the left of $\gamma^{\ell-2}$. 
Here we exploit the fact that pairs in $J$ are dangerous: 
Consider a walk $\beta$ witnessing the fact that $a^{\ell-1}$ sees the left side of $x_0$. 
Since $(a^{\ell-2}, b^{\ell-2})$ is a special pair of $C^{\ell-2}$, the element $a^{\ell-1}$ is drawn above $a^{\ell-2}$ and is moreover to the right of $\gamma^{\ell-2}$; this situation is illustrated in Figure~\ref{fig:wrong-x0}.
Thus we deduce that $\beta$ intersects $\gamma^{\ell-2}$. 
However, this implies that $a^{\ell-1} \leq b^{\ell-1}$ in $P$, a contradiction.
This concludes the proof.
\begin{figure}[h]
	\centering
	\includegraphics[scale=1.0]{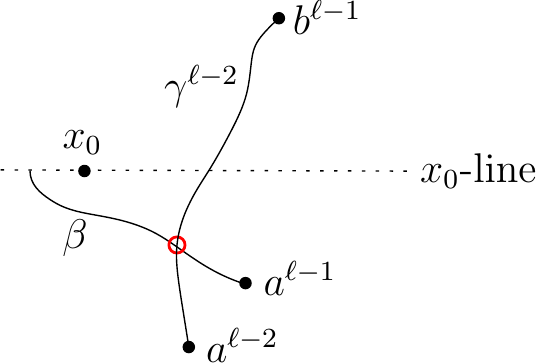}
	\caption{Situation if we assume that $x_0$ is to the left of $\gamma^{\ell-2}$. Considering a walk from $a^{\ell-1}$ to the left side of $x_0$, we obtain a forbidden comparability in $P$.}
	\label{fig:wrong-x0}
\end{figure}
\end{proof}

Hence, we conclude that $\dim(I_4) \leq 3h$ as claimed.   
This finishes the proof of Lemma~\ref{lemma:point-below-max}.

\section{More tools}\label{sec:tools}
In this section we introduce additional tools that will be used in the main proof. 
\subsection{Unfolding a poset}\label{sec:unfolding}
First, let us introduce a couple standard definitions that will be used in this section. 
A poset is said to be \emph{connected} if its cover graph is connected. 
An induced subposet $Q$ of a poset $P$ is a {\em convex} subposet of $P$ if whenever we have $x<y<z$ in $P$ and $x,z \in Q$, then $y$ is also included in $Q$.
Given a subset $X$ of elements in $P$, the \emph{convex hull} of $X$ in $P$ is the subposet of $P$ induced by the set of elements  $x\in P$ for which there are $y,z\in X$ such that $y\leq x\leq z$ in $P$.
We denote this subposet by $\conv_P(X)$ (or simply by $\conv(X)$ if $P$ is clear from the context).
Note that $\conv_P(X)$ is a convex subposet of $P$, and that it is equivalently defined by taking the smallest convex subposet of $P$ containing $X$.

The {\em dual poset} $P^d$ of a poset $P$ is the poset on the same set of elements with $x \leq y$ in $P^d$ if and only if $y\leq x$ in $P$.
It is immediate that $\dim(P) = \dim(P^d)$, and what is more important to us, $\dim_P(\Min(P),\Max(P)) = \dim_{P^d}(\Min(P^d),\Max(P^d))$. 
This observation will be used in the next section. 

If $Q$ is an induced subposet of $P$, then 
\begin{equation}\label{eq:dim-for-induced-subposets}
\dim_Q(I) = \dim_P(I)
\end{equation}
for every $I \subseteq \Inc(Q)$. 
This follows from the fact that $C \subseteq I$ forms an alternating cycle in $Q$ if and only if $C$ forms an alternating cycle in $P$. 
As a result, we do not need to specify whether the ambient poset is $P$ or $Q$ when considering the dimension of the set $I$, and we often simply write $\dim(I)$ in such a situation.     

Let $P$ be a connected poset with at least two elements and let $A:=\Min(P)$ and $B:=\Max(P)$. 
Observe that $A$ and $B$ are disjoint in this case. 
Given a fixed element $x_0\in A \cup B$ we can \emph{unfold} the poset $P$ starting from $x_0$ in the following way.
If $x_0\in A$ then set $A_0:=\set{x_0}$ and $B_1:= \{b\in B\mid a_0 \leq b\text{ in } P\}$. 
If $x_0\in B$ then set $A_0:=\emptyset$ and $B_1:= \{x_0\}$.  
In both cases, define the following sets for $i=1,2,\ldots$
 \begin{align*}
  A_i&:=\Big\{a\in A- \bigcup_{0\leq j<i}A_j\mid \text{ there is }b\in B_i\text{ with }a\leq b\text{ in }P\Big\}, \\
  B_{i+1}&:=\Big\{ b\in B- \bigcup_{1\leq j<i+1}B_j\mid \text{ there is }a\in A_{i}\text{ with }a\leq b\text{ in }P\Big\}.
 \end{align*}

Since $P$ is connected, the sets $A_0,A_1,\ldots$ partition $A$, and the sets $B_1,B_2\ldots$ partition $B$. 
Let us emphasize that $A_0$ is possibly empty. 
Note that after some point all sets in the sequence $A_0,B_1,A_1,B_2\ldots$ are empty. 
A prefix of the sequence $A_0,B_1,A_1,B_2\ldots$ containing all non-empty sets is called an \emph{unfolding} of $P$ from $x_0$. 
See Figure~\ref{fig:zigzag} for an illustration. 
Let us point out the following simple but important observation: 

\begin{equation}
\label{unfolding-properties}
\parbox{0.8\textwidth}{
If $x \in \Up(A_i)$ with $i\geq 1$, then $x$ is in at least one of the two downsets $\D(B_i)$ and $\D(B_{i+1})$, possibly both, but in no other downset $\D(B_j)$ with $j\neq i, i+1$. 
If $x \in \Up(A_0)$ then  $x$ is included only in $\D(B_1)$. \\

Dually, if $x \in \D(B_i)$  with $i\geq 1$, then $x$ is in at least one of the two upsets $\Up(A_{i-1})$ and $\Up(A_i)$, possibly both, but in no other upset $\Up(A_j)$ with $j\neq i-1, i$.
}
\end{equation}

The idea of unfolding a poset was introduced by Streib and Trotter~\cite{ST14} and was subsequently used in several works~\cite{JMTWW,MW,JMW-sparsity}.
The following lemma plays a key role in all applications of unfolding. 
Intuitively, it says that every unfolding contains a local part that roughly witnesses the dimension of the poset (more precisely, $\dim(A,B)$), up to a factor $2$.
We refer the reader to e.g.~\cite{MW} for a proof.

\begin{lemma}[Unfolding lemma]\label{lem:unfolding-chi}
 Let $P$ be a connected poset, let $A:=\Min(P)$, $B:=\Max(P)$, and suppose that $\dim(A,B)\geq 2$.
 Consider the sequence $A_0,B_1,\ldots,A_{m-1},B_m$ obtained by unfolding $P$ from some element $x_0\in A\cup B$. 
 Then there exists an index $i$ such that
 \begin{align*}
  \dim(A_{i},B_{i}) &\geq \dim(A,B)/2 \\
   & \mathrm{\it or} \\
  \dim(A_{i},B_{i+1}) &\geq \dim(A,B)/2 
 \end{align*}
holds. 
\end{lemma}

Suppose that we have $\dim(A_i,B_j)\geq \dim(A,B)/2$ for some $j\in \{i,i+1\}$, as in the lemma. 
Consider the convex subposet $Q:=\conv(A_i\cup B_j)$ whose element set is $(\Up(A_i)\cap \D(B_j))\cup A_i\cup B_j$ as follows from the definition of a convex hull.
Recall that by~\eqref{eq:dim-for-induced-subposets} we have $\dim_P(A_i,B_j) = \dim_{Q}(A_i,B_j)$, thus we can (and do) omit subscripts in the above lemma. 

Assume further that $\dim(A,B) \geq 6$, implying $\dim(A_i,B_j)\geq 3$, which will be the case when we apply the lemma later on in the proof. 
While our starting poset $P$ was assumed to be connected, it could be that $\conv(A_i\cup B_j)$ is not. 
Since we would like to keep connectivity in what follows, we use the following observation: 
If $R$ is a disconnected poset and $\dim(\Min(R), \Max(R)) \geq 3$, then there is a component $R'$ of $R$ with $\dim(\Min(R'), \Max(R'))=\dim(\Min(R), \Max(R))$. 
(This is not difficult to prove, see e.g.~\cite[Observation~5]{JMTWW} for a proof).     
Going back to our poset $\conv(A_i\cup B_j)$, since we assumed that $\dim(A_i,B_j) \geq 3$, this means that there is a component $P'$ of $\conv(A_i\cup B_j)$ with $\Min(P')\subseteq A_i$ and $\Max(P')\subseteq B_j$ such that
\[
 \dim(\Min(P'),\Max(P'))=\dim(A_i,B_j) \geq \dim(A,B)/2.
\]
The poset $P'$ is then said to be a \emph{core} of $P$ with respect to $x_0$, or simply an {\em $x_0$-core} of $P$ in brief.
Clearly, $P'$ is a convex subposet of $P$. 
Moreover, since $P'$ is connected, the set of elements of $P'$ is precisely $\Up_P(\Min(P')) \cap \D_P(\Max(P'))$. 

We attribute a type to $P'$ depending on whether $j=i$ or $j=i+1$:  
If $j=i$ then $P'$ is \emph{left-facing}, and if $j=i+1$ then $P'$ is \emph{right-facing}.
For example, under the assumption that $\conv(A_2\cup B_2)$ in Figure~\ref{fig:zigzag-claims} on the left is a core, $\conv(A_2\cup B_2)$ would be left-facing.

\begin{figure}[t]
  \centering
  \includegraphics[scale=1.0]{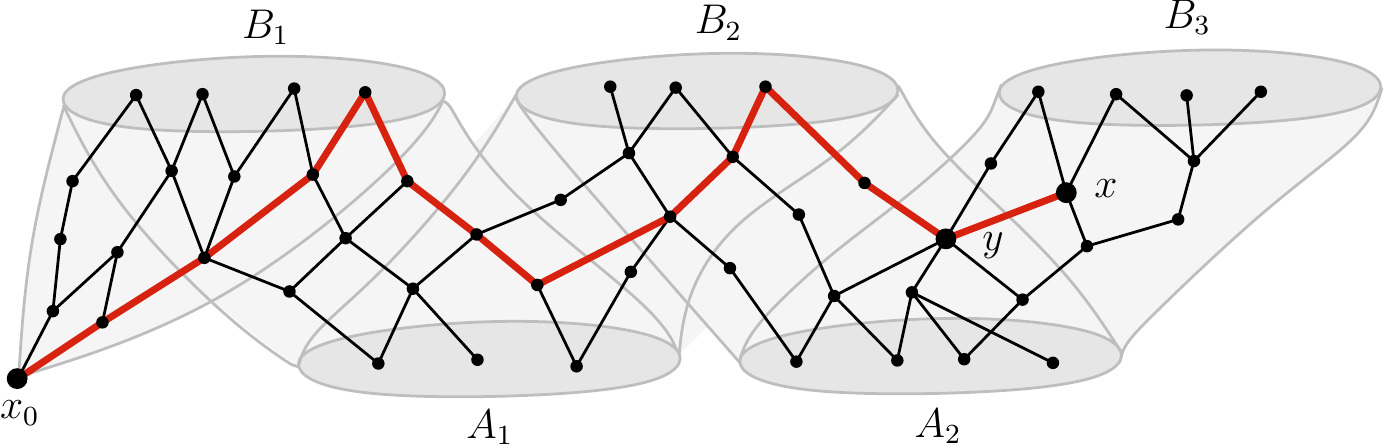}
  \caption{Unfolding $P$ from $a_0$ and the zig-zag path $\Z_P^{a_0}(x)$ of $x$.}
  \label{fig:zigzag}
\end{figure}
 
To summarize, when $P$ is a connected poset with $\dim(\Min(P),\Max(P)) \geq 6$, then we may consider a core $P'$ of $P$, and such a core has the following properties: 
\begin{enumerate}
\item $P'$ is connected;
\item $P'$ is a convex subposet of $P$ with element set $\Up_P(\Min(P')) \cap \D_P(\Max(P'))$;
\item $\dim(\Min(P'),\Max(P')) \geq \dim(\Min(P),\Max(P))/2$, and
\item $P'$ is either left-facing or right-facing.
\end{enumerate}

\subsection{Zig-zag paths}
Let $P$ be a connected poset and let $A:=\Min(P)$ and $B:=\Max(P)$.
Choose some $x_0\in A\cup B$ and let $A_0,B_1,\ldots,A_{m-1},B_m$ be the sequence obtained by unfolding $P$ starting from $x_0$.
We will define for each element $x\in P$ a corresponding `zig-zag path' connecting $x$ with $x_0$ in the cover graph of $P$.

To do so, we first need to introduce some notations.  
For $x\in P-\set{x_0}$, let $\alpha(x)$ denote the smallest index $i\geq 0$ such that $x\in\Up(A_i)$, and $\beta(x)$ the smallest index $j\geq 1$ such that $x\in \D(B_j)$. 
(For instance, in Figure~\ref{fig:zigzag} we have $\alpha(x)=2, \beta(x)=3$ and $\alpha(y)=2,\beta(y)=2$.)
Note that $x\in \conv(A_{\alpha(x)}\cup B_{\beta(x)})$ and $\beta(x)\in\set{\alpha(x),\alpha(x)+1}$ by \eqref{unfolding-properties}.

Next, we associate to each element $x\in P-\set{x_0}$ a {\em parent}: 
If $\beta(x)=\alpha(x)=i$, then note that $x\not\in B_i$ since $\alpha(b)=i-1$ for all elements $b\in B_i$.
Thus, there exist $y\in P$ and $b\in B_i$ such that $x<y\leq b$ in $P$ and $x<y$ is a cover relation in $P$, and we choose (arbitrarily) one such element $y$ to be the parent of $x$.  
If $\beta(x)=\alpha(x)+1=i+1$ (as in Figure~\ref{fig:zigzag}) then $x\not\in A_i$ since $\beta(a)=i$ for all elements $a\in A_i$. 
Thus, there exist $y\in P$ and $a\in A_i$ such that $a\leq y<x$ in $P$ and $y<x$ is a cover relation.
We choose one such element $y$ to be the parent of $x$. 

We write $\parent(x)$ to denote the parent of $x$. 
Observe that when $\alpha(x)=\beta(x)=i$, then by definition $\parent(x) \in \D(B_i)$ and hence $\alpha(\parent(x)) \leq i = \alpha(x)$ by~\eqref{unfolding-properties}.
Also since $\parent(x) \in \D(B_i)$, we have $\beta(\parent(x)) \leq i = \beta(x)$. 
Similarly,  when $\alpha(x)=i$ and $\beta(x)=i+1$, then by definition $\parent(x) \in \Up(A_i)$ and hence $\alpha(\parent(x)) \leq i = \alpha(x)$.
And since $\parent(x) \in \Up(A_i)$, we also have $\beta(\parent(x)) \leq i+1 = \beta(x)$ 
by~\eqref{unfolding-properties}.

To summarize, we have: 
\begin{equation}
\label{monotonicity-of-alpha-and-beta}
\alpha(\parent(x)) \leq \alpha(x) \text{ and } \beta(\parent(x)) \leq \beta(x),
\end{equation}
for every $x\neq x_0$ in $P$.

Let $T$ be the spanning subgraph of $\cover(P)$ obtained by only keeping edges connecting an element and its designated parent. 

\begin{claim} 
\label{claim:tree}
 The graph $T$ is a tree.
\end{claim}
\begin{proof}
To see this, it is convenient to orient each edge $\set{x,\parent(x)}$ of $T$ towards $\parent(x)$.
Then every element of $P$ is incident to exactly one outgoing edge, except for $x_0$ which is a sink. 
Thus, to show that $T$ is a tree it is enough to show that there is no directed cycle in this orientation of $T$. 
Arguing by contradiction, suppose that $x_1x_2 \dots x_k$ is a directed cycle. 
Then $\alpha(x_1) \leq \cdots \leq \alpha(x_k) \leq \alpha(x_1)$ and $\beta(x_1) \leq \cdots \leq \beta(x_k) \leq \beta(x_1)$ by~\eqref{monotonicity-of-alpha-and-beta}. 
Thus all these inequalities hold with equality. 
By~\eqref{unfolding-properties} we know that $\beta(x_i) \in \set{\alpha(x_i),\alpha(x_i)+1}$, for all $i\in\{1,\ldots,k\}$. 
Suppose first that $\beta(x_i)=\alpha(x_i)$, for all $i\in\{1,\ldots,k\}$. 
Then $x_i < \parent(x_i) = x_{i+1}$ in $P$ for each $i\in\{1,\dots,k\}$ (cyclically), which is a contradiction. 
Now assume that $\beta(x_i)=\alpha(x_i)+1$, for all $i\in\{1,\ldots,k\}$. 
Then $x_i > \parent(x_i) = x_{i+1}$ in $P$ for each $i\in\{1,\dots,k\}$ (cyclically), which is again a contradiction.
This completes the proof.
\end{proof}

For each $x\in P$ let $\Z_P^{x_0}(x)$ denote the unique path in $T$ that connects $x_0$ and $x$, which we call the {\em zig-zag path} of $x$. 
(We drop the subscript $P$ when the poset is clear from the context.)  
See Figure~\ref{fig:zigzag} for an illustration. 

The following claim describes the shape of the initial part of the zig-zag path $Z^{x_0}(y)$ starting from $y$.
It is easy to see that when $y$ belongs to all three sets $\Up(A_{i-1})$, $\Up(A_i)$, and $\D(B_i)$, then the zig-zag path goes down from $y$ in $P$.
We will see that the trace of the zig-zag path $Z^{x_0}(y)$ in $\Up(A_i)$ is a chain, while its trace  in $\D(B_i)$ has a unique minimal element, see Figure~\ref{fig:zigzag-claims} for an illustration.

\begin{figure}[t]
  \centering
  \includegraphics[scale=1.0]{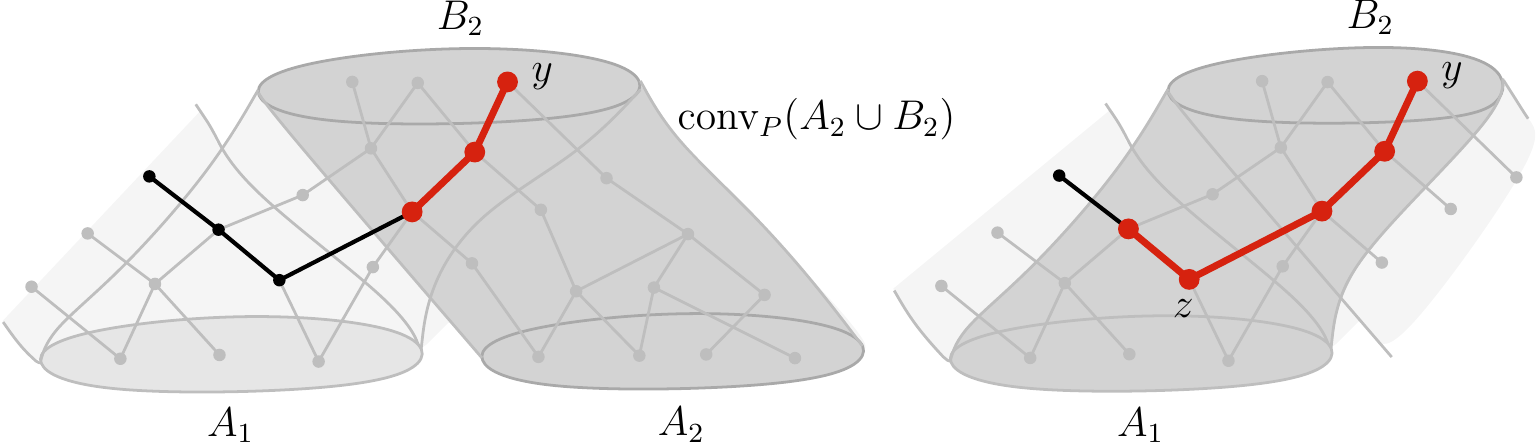}
  \caption{Illustration of Claim~\ref{claim:zigzag-chain}.}
  \label{fig:zigzag-claims}
\end{figure}

\begin{claim}\label{claim:zigzag-chain}
	If an element $y$ lies in all three sets $\Up(A_{i-1})$, $\Up(A_i)$, and $\D(B_i)$ for some $i\in\set{1,\ldots,k}$, then
	\begin{enumerate}
		\item the elements of $\Z^{x_0}(y)\cap \Up(A_i)$ form a chain contained in $\Up(A_{i-1})$, and $y$ is the maximal element of that chain;\label{item:zig-zag-chain}
		\item  the elements of $\Z^{x_0}(y) \cap \D(B_i)$ induce a subposet of $P$ with a unique minimal element.\label{item:zig-zag-unique-min}
	\end{enumerate}
\end{claim}

\begin{proof}
Since $y$ belongs to $\Up(A_{i-1})$, $\Up(A_i)$, and $\D(B_i)$, we see that $\alpha(y) = i-1$ and $\beta(y) = i$ by~\eqref{unfolding-properties}.

Now start in $y$ and keep walking along the zig-zag path  $\Z^{x_0}(y)$ towards $x_0$, as long as the parent of the current element $x$ is smaller than $x$ in $P$, or we reach $x_0$. 
Let $z$ be the element we stop at.
If $z=x_0$ then the whole zig-zag path $Z^{x_0}(y)$ is simply a chain in $P$ with $y$ being its maximal element, thus~\ref{item:zig-zag-chain} holds.
Property~\ref{item:zig-zag-unique-min} holds as well as the whole chain lies in $D(B_1)$ and has a unique minimal element, namely $x_0$.

So we assume that $z\neq x_0$.
In this case, we keep walking along the zig-zag path $\Z^{x_0}(y)$ towards $x_0$, as long as the parent of the current element $x$ is above $x$ in $P$, or we reach $x_0$.  
Let $z'$ be the new element we stop at.

Observe that $z < \parent(z)$ in $P$. 
Since $z \leq y$ in $P$ and $y \in \D(B_i)$, we have $z \in D(B_i)$. 
Clearly, $z$ is smaller or equal (in $P$) to all the other elements of the zig-zag path $Z^{x_0}(y)$ that we traversed.
We have $\alpha(z) \leq \alpha(y)=i-1$, by~\eqref{monotonicity-of-alpha-and-beta}.
We also have $\alpha(z) \geq \alpha(y)$ since $z \leq y$ in $P$ and thus $\D(z) \subset \D(y)$.
Hence $\alpha(z) = \alpha(y)=i-1$. 
Since $z < \parent(z)$ in $P$, we must have $\beta(z)=\alpha(z)=i-1$ by the definition of parents.

It follows that $\beta(x) \leq \beta(z) = i-1$ for all $x$ on the zig-zag path $Z^{x_0}(y)$ appearing after $z$ (towards $x_0$) as well, by~\eqref{monotonicity-of-alpha-and-beta}.  
For such elements $x$ we cannot have $x \in U(A_i)$ by~\eqref{unfolding-properties}. 
Hence $\Z^{x_0}(y) \cap \Up(A_i)$ is a subset of the $y$--$z$ portion of the zig-zag path  $\Z^{x_0}(y)$. Thus it is a chain with $y$ as its maximal element.
This proves~\ref{item:zig-zag-chain}. 

For the proof of~\ref{item:zig-zag-unique-min}, we are going to show that $z$ is the unique minimal element in the subposet of $P$ induced by $\Z^{x_0}(y) \cap \D(B_i)$.
Recall that $z\in \D(B_i)$.

If $z'=x_0$ then all elements on the zig-zag path  $\Z^{x_0}(y)$ are greater or equal to $z$ in the poset $P$. 
Thus~\ref{item:zig-zag-unique-min} holds.

Now, assume that $z'\neq x_0$. 
Thus $z' > \parent(z')$ in $P$. 
We have $\beta(z') \leq \beta(z)$, by~\eqref{monotonicity-of-alpha-and-beta}.
We also have $\beta(z') \geq \beta(z)$ since $z' > z$ in $P$ and thus $\Up(z') \subset \Up(z)$.
Hence $\beta(z') = \beta(z)=i-1$. 
Since $z' > \parent(z')$ in $P$, we must have $\alpha(z')=\beta(z')-1=i-2$ by the definition of parents.

It follows that $\alpha(x) = \alpha(z') \leq i-2$ for all $x$ on the zig-zag path $Z^{x_0}(y)$ appearing after $z'$ (towards $x_0$) as well, by~\eqref{monotonicity-of-alpha-and-beta}.  
For such elements $x$ we cannot have $x \in D(B_i)$ by~\eqref{unfolding-properties}. 
Hence $\Z^{x_0}(y) \cap \D(B_i)$ is a subset of the $y$--$z'$ portion of the zig-zag path  $\Z^{x_0}(y)$.  
Clearly, $z$ is the unique minimal element in the subposet induced by that subset.
This completes the proof of~\ref{item:zig-zag-unique-min}.
\end{proof}

\section{Proof of the main theorem}\label{sec:main-proof}
Let $P$ be a planar poset of height $h$.
First, we apply Lemma~\ref{lem:min-max-reduction} to obtain another planar poset $P'$ of height $h$ such that $\dim(P)\leq \dim(\Min(P'), \Max(P'))$.
In what follows we consider only the poset $P'$, so with a slight abuse of notation let us simply write $P$ for $P'$ from now on.
Our aim is to show that $\dim(\Min(P), \Max(P)) \leq 192h+96$, which implies our main theorem.

Clearly, we may assume that $\dim(\Min(P), \Max(P)) > 192 + 96 = 288$, as otherwise we are done. 
As noted in Section~\ref{sec:tools} we may assume that $P$ is connected.

Fix a plane drawing of the diagram of $P$ and let $G$ denote the cover graph of $P$ embedded in the plane according to this drawing. 
For simplicity, we may assume without loss of generality that no two elements of $P$ have the same $y$-coordinate in the drawing. 
We remark however that it is not essential to the arguments developed in this section that the diagram itself can be drawn in planar way, just that the cover graph can.     
Planarity of the diagram itself will be needed only when invoking Lemma~\ref{lemma:point-below-max}. 

The following straightforward observation will be used several times implicitly in the proof:
If $Q$ is a convex subposet of $P$ then our fixed drawing of the diagram of $P$ induces a plane drawing of the diagram of $Q$.

\subsection{Unfolding the poset}
\label{sec:unfolding_the_poset}
Given a connected convex subposet $Q$ of $P$, let $a(Q)$ denote the element of $Q$ with smallest $y$-coordinate in the drawing.
Clearly, $a(Q)$ is a minimal element of $Q$. 
Similarly, let $b(Q)$ denote the element of $Q$ with largest $y$-coordinate in the drawing, which is a maximal element of $Q$.

Using the terminology introduced in Section~\ref{sec:unfolding}, we iteratively unfold the poset $P$ three times: 
Let $Q_0 := P$. 
For $i=1,2,3$ let $c_{i-1} := a(Q_{i-1})$ and let $Q_{i}$ be a $c_{i-1}$-core of $Q_{i-1}$.  
Note that
\[
\dim(\Min(Q_2), \Max(Q_2)) \geq \dim(\Min(Q_{1}), \Max(Q_{1}))/2 \geq \dim(\Min(P), \Max(P))/4 \geq 6, 
\]
thus these cores can be defined (see the end of Section~\ref{sec:unfolding}). 

At least two of the cores $Q_1, Q_2, Q_3$ have the same type (left-facing or right-facing).  
Say this is the case for indices $i$ and $j$, with $i<j$. 
We would like to focus on the left-facing case in the rest of the proof. 
We will be able to do so thanks to the following trick:  
Say $Q_i$ and $Q_j$ are right-facing. 
Then we turn our attention to the dual poset $P^d$ of $P$ and its drawing obtained from that of $P$ by flipping it. 
For $k=1,2,3$, unfold $Q^d_{k-1}$ from $c_{k-1}=b(Q^d_{k-1})$; notice that $Q^d_{k}$ is a core of the unfolding, as follows from these two observations: 
\begin{enumerate}
    \item if $\{x\}, B_1, A_1, \dots, A_{m-1}, B_m$ is an unfolding of a poset $R$ with $x\in\Min(R)$, then $\emptyset, \{x\}, B_1, A_1, \dots, A_{m-1}, B_m$ is an unfolding of the dual poset $R^d$; 
    \item given a poset $R$ and $A\subseteq \Min(R)$, $B\subseteq \Max(R)$ we have $\dim_{R}(A,B) = \dim_{R^d}(B,A)$. 
\end{enumerate}
In other words, we simply mirror the three unfoldings we did before. 
The key observation is that the type of the core  $Q^d_{k}$ is the opposite of that of $Q_{k}$. 
In particular, $Q^d_i$ and $Q^d_j$ are left-facing.  
We may then work with $P^d$ and  $Q^d_i, Q^d_j$ instead of $P$ and $Q_i, Q_j$. 

To summarize, going to the dual poset if necessary, we may assume that $Q_i$ and $Q_j$ are both left-facing. 
Then $c_{i-1}$ is equal to either $a(Q_{i-1})$ or $b(Q_{i-1})$---i.e.\ $c_{i-1}$ is either the bottommost or the topmost point in the drawing of $Q_{i-1}$---and the same holds for $c_{j-1}$ w.r.t.\ $Q_{j-1}$.

Let $s_1:=c_{i-1}$ and $s_2:=c_{j-1}$. 
(Thus, either $s_1, s_2 \in A$ or $s_1, s_2 \in B$.) 
Let $P^1 := Q_{i-1}$ and $P^2 := Q_{j-1}$. 
Also, let $P^3:= Q_j$ and $A^3:=\Min(P^3)$, $B^3:=\Max(P^3)$.

By Lemma~\ref{lem:unfolding-chi} and the definition of cores we have:  
\[
 \dim(A^3,B^3)\geq \dim(\Min(P),\Max(P))/8.
\]
Our goal is to show $\dim(A^3,B^3) \leq 24h+12$, which implies our main theorem. 

Before pursuing further, let us emphasize that $P^3, P^2, P^1$ form an increasing sequence of {\em convex} subposets of $P$. 
In particular, they all are planar posets, and the drawing of $P$ induces in a natural way drawings of their respective diagrams. 
It is perhaps good to also recall that we do not need to specify whether $\dim(A^3,B^3)$ is to be understood w.r.t.\ $P^3, P^2, P^1$,  or $P$ as this is the same quantity.

Let $A_0^1,B_1^1,A_1^1,\ldots,A_{m_1-1}^1,B_{m_1}^1$ and $A_0^2,B_1^2,A_1^2,\ldots,A_{m_2-1}^2,B_{m_2}^2$ be sequences obtained by unfolding $P^1$ and $P^2$ from $s_1$ and $s_2$, respectively.
It follows from our left-facing assumption that there are indices $k\geq 1$ and $\ell\geq 1$ such that $P^2$ is contained in $\conv_{P^1}(A^1_k\cup B^1_k)$ and $P^3$ is contained in $\conv_{P^2}(A^2_{\ell}\cup B^2_{\ell})$. 
Figure~\ref{fig:iterative-unfold} illustrates this in the case $k=\ell=2$. 

\begin{figure}[t]
 \centering
 \includegraphics[scale=1.0]{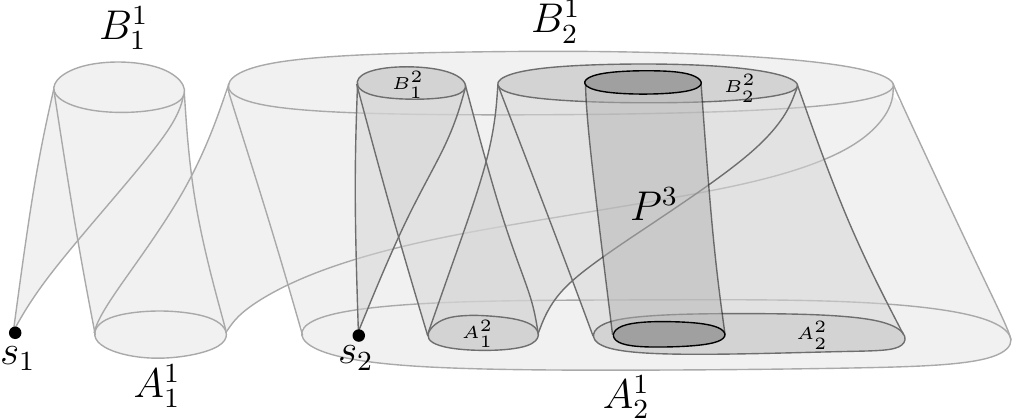}
 \caption{Unfolding $P^1$ and $P^2$.}
 \label{fig:iterative-unfold}
\end{figure}

Let us point out that $s_1$ is not in $P^2$. 
This can be seen as follows: 
First, note that either $A_0^1 = \{s_1\}$, or $A_0^1 = \emptyset$ and $B_1^1 = \{s_1\}$. 
In the first case, it is clear that $s_1$ is not in $P^2$ simply because $s_1$ will not be included in $\D(B_k^1) \cap \Up(A^1_k)$. 
In the second case, $s_1$ is not in $P^2$ because $\Inc(A_1^1, B_1^1) = \emptyset$, thus $\dim(A_1^1, B_1^1) = 1$, and hence $k \geq 2$ in that case. 
The same observation holds for $s_2$ w.r.t.\ $P^3$. 
Let us emphasize this observation:  
\begin{equation}
s_1 \notin P^2 \quad \textrm{ and } \quad s_2 \notin P^3. 
\label{eq:notin}
\end{equation}

\subsection{Red tree and blue tree}
Let $G^3$ denote the cover graph of $P^3$ embedded in the plane according to our fixed drawing of $P$.  
Since neither $s_1$ nor $s_2$ is in $P^3$ (c.f.~\eqref{eq:notin}) and both elements have either smallest or largest $y$-coordinates among elements of $P^1$ and $P^2$, respectively, we see that $s_1$ and $s_2$ are each drawn either below all elements of $G^3$ or above all of them. 
In particular, $s_1$ and $s_2$ are drawn in the outer face of $G^3$.

We will define two specific subgraphs of $G$, the cover graph of $P$.   
Both subgraphs will be trees and will have the property that all their internal nodes are drawn in the outer face of $G^3$ and all their leaves are drawn on the boundary of that outer face. 
The first tree is rooted at $s_1$ and will be colored red, while the second one is rooted at $s_2$ and will be colored blue.
See Figure~\ref{fig:red-blue-paths} for an illustration of how the trees will look like.
To define the trees we will use zig-zag paths with respect to our fixed unfoldings of $P^1$ and $P^2$.
For simplicity, given an element $v\in P^3$ we denote by $Z^1(v)$ and $Z^2(v)$ the zig-zag paths $\Z_{P^1}^{s_1}(v)$ and $\Z_{P^2}^{s_2}(v)$, respectively.

Now, consider an element $b\in B^3$ and its zig-zag path $\Z^1(b)$ in $P^1$ connecting $b$ to $s_1$. 
Let $x$ be the element of $P^3$ on that path that is closest to $s_1$. 
Recall that $s_1$ is not in $P^3$ by~\eqref{eq:notin}, thus $x\neq s_1$.    
Since $P^3$ is a core of $P^2$, and in particular the element set of $P^3$ is $\Up_{P^2}(A^3) \cap \D_{P^2}(B^3)$, 
there is an element $a\in A^3$ such that $a \leq b$ in $P^3$. 
This shows that $b \in \Up_{P^1}(A_k^1)$, since $a \in A^3 \subseteq A_k^1$. 
We also know that $b \in B^3 \subseteq B_k^1\subseteq \Up_{P^1}(A_{k-1}^1)$ (by definition of unfolding).  
Hence, we may apply Claim~\ref{claim:zigzag-chain}\ref{item:zig-zag-chain} on the poset $P^1$ and element $b$ w.r.t.\ the unfolding sequence $A_0^1,B_1^1,A_1^1,\ldots,A_{m_1-1}^1,B_{m_1}^1$ of $P^1$. 
By this claim, we know that $Z^{1}(b)\cap \Up_{P^1}(A_k^1)$ is a chain $C$ with the topmost element being $b$ and $C \subseteq \Up_{P^1}(A_{k-1}^1)$.
Observe in particular that $x$ is in $C$ and hence $x\leq b$ in $P^3$.

No element of the zig-zag path $Z^1(x)$ is in $P^3$ except for $x$, by the choice of $x$.   
Thus, since $s_1$ is drawn in the outer face of $G^3$, it follows that the whole path $Z^1(x)$ is drawn in the outer face, except for its endpoint $x$ which is on its boundary. 
We call $x$ the \emph{red exit point} of $b$, and the edge of $Z^1(x)$ incident to $x$ the {\em red exit edge} of $b$. 
Since $x$ is in $P^3$ and in the chain $C \subseteq \Up_{P^1}(A_{k-1}^1)$ we obtain the following obvious relations, which we emphasize for future reference: 

\begin{align}
x  \in  \D_{P^1}(B_k^1), \quad  \quad x \in  \Up_{P^1}(A_k^1), \quad \textrm{ and } \quad  x \in  \Up_{P^1}(A_{k-1}^1). 
\label{eq:red-exit-points}
\end{align}

Let $X$ be the set of red exit points of elements in $B^3$.
Let $T^1:=\bigcup_{x\in X}\Z^1(x)$. 
Then $T^1$ is a tree with $X$ as set of leaves, as follows from Claim~\ref{claim:tree}. 
We refer to $T^1$ as the \emph{red tree}.
Note that in our drawing of $\cover(P)$, the red tree is drawn in the outer face of $G^3$ with its leaves on the boundary, as in Figure~\ref{fig:red-blue-paths}.  

By considering $P^2$ instead of $P^1$ in the definition of red exit points, we  analogously define blue exit points for elements in $B^3$:
Given $b\in B^3$, the {\em blue exit point} of $b$ is the element $y$ of $P^3$ on the zig-zag path $Z^2(b)$ that is closest to $s_2$.   
The edge of $Z^2(y)$ incident to $y$ is said to be the {\em blue exit edge} of $b$. 
The element $y$ has the following properties: 
\begin{align}
y  \in  \D_{P^2}(B_\ell^2), \quad \quad y \in  \Up_{P^2}(A_\ell^2), \quad \textrm{ and } \quad  y \in  \Up_{P^2}(A_{\ell-1}^2). 
\label{eq:blue-exit-points}
\end{align}
 
We let $Y$ denote the set of blue exit points of elements in $B^3$, and let $T^2:=\bigcup_{y\in Y}\Z^2(y)$ denote the tree defined by the union of the zig-zag paths of these exit points.  
$Y$ is thus the set of its leaves, and the tree is drawn in the outer face of $G^3$, except for its leaves which are on its boundary. 
We refer to $T^2$ as the {\em blue tree}.    

\begin{figure}[t]
 \centering
 \includegraphics[scale=1.0]{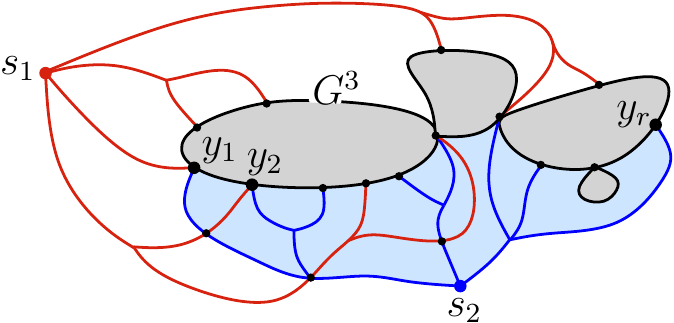}
 \caption{The graph $G^3$ together with the red and blue trees. 
The part outside $G^3$ is not drawn according to our fixed embedding of the diagram but has instead been redrawn freely for clarity, relying only on the planarity of the cover graph.}
 \label{fig:red-blue-paths}
\end{figure}

\subsection{Red exit edges trapped in the blue tree}
\label{sec:redblue_intersect}
Let $H$ denote the plane graph obtained by taking the union of $G^3$ and the blue tree. 
We use $R_H$ to denote the region of the plane bounded by the outer face of $H$ (including the boundary). 
In this section we deal with elements $b \in B^3$ whose red exit edges are ``trapped'' in the blue tree, in the sense that they are drawn in the region $R_H$.  
We will show that incomparable pairs involving these $b$'s can be partitioned into at most $12h+6$ reversible sets. 

Since the blue tree is embedded in the plane, the clockwise orientation of the plane induces a cyclic ordering of its leaves. 
Enumerate these leaves as $y_1,\ldots,y_r$ according to this ordering, in such a way that the two zig-zag paths $Z^2(y_1)$ and $Z^2(y_r)$ are on the boundary of the outer face of $H$.  
Observe that for each $i\in \{2, \dots, r-1\}$, the edge $uy_i$ of $Z^2(y_i)$ incident to $y_i$ is drawn in $R_H$. 

Let $B^{\mathrm{trapped}}$ be the set of all elements $b\in B^3$ such that the red exit edge of $b$ is drawn in $R_H$. 

\begin{claim}
\label{claim:trapped}
$\dim(A^3,B^{\mathrm{trapped}})\leq 12h+6$.
\end{claim}
\begin{proof}
We will show that there are two elements $z_1$ and $z_r$ of $P_2$ such that, for each element $b\in B^{\mathrm{trapped}}$, we have $z_1 \leq b$ or $z_r \leq b$ in $P_2$.
Thus, $\Inc(A^3,B^{\mathrm{trapped}}) \subset \Inc(A^3,B^3 \cap \Up_{P^2}(z_1))\cup \Inc(A^3,B^3 \cap \Up_{P^2}(z_r))$.
Since $\dim(A^3, B^3 \cap \Up_{P^2}(z_i)) \leq 6h+3$, for $i\in\set{1,r}$ by Lemma~\ref{lemma:point-below-max}, this will conclude the proof of the claim.

To do so, the two zig-zag paths $\Z^2(y_1)$ and $\Z^2(y_r)$ will play an important role.
Recall that $y_1,y_r \in  \D_{P^2}(B^2_{\ell})$, $y_1,y_r \in \Up_{P^2}(A^2_{\ell})$, and $y_1,y_r \in \Up_{P^2}(A^2_{\ell-1})$ by~\eqref{eq:blue-exit-points}. 
Apply Claim~\ref{claim:zigzag-chain}\ref{item:zig-zag-unique-min} on $P^2$ with element $y_1$.  
By this claim, there is an element $z_1$ of $P^2$ such that $\Z^2(y_1)\cap \D_{P^2}(B^2_{\ell})\subseteq \Up_{P^2}(z_1)$. 
Doing the same with element $y_r$, we find an element $z_r$ of $P^2$ such that $\Z^2(y_r)\cap \D_{P^2}(B^2_{\ell})\subseteq \Up_{P^2}(z_r)$. 
Let us repeat the properties of $z_1$ and $z_r$, for future reference: 
\begin{align}
 \Z^2(y_1)\cap \D_{P^2}(B^2_{\ell})\subseteq \Up_{P^2}(z_1)\quad\text{ and }\quad \Z^2(y_r)\cap \D_{P^2}(B^2_{\ell})\subseteq \Up_{P^2}(z_r).\label{eq:y-points}
\end{align}
Or in words: In $P^2$, the element $z_1$ is smaller or equal to each element of the zig-zag path $Z^2(y_1)$ that is in the downset of $B^2_{\ell}$, and the same holds for $z_r$ with respect to $y_r$.

Consider $b\in B^{\mathrm{trapped}}$. 
Let $x$ denote its red exit point, for which we have $x\leq b$ in $P^3$, and $ux$ its red exit edge.
Observe that the endpoint $s_1$ of the path $Z^1(b)$ is not in $R_H$, since $s_1$ is either topmost or bottommost in the drawing of $P^1$. 
On the other hand, the edge $ux$ is drawn in the region $R_H$. 
Recall that every vertex on the boundary of the outer face of $H$ that is not in $G^3$ is in the union of $Z^2(y_1)$ and $Z^2(y_r)$. 
Thus, we see that $Z^1(b)$ must intersect $Z^2(y_1)$ or $Z^2(y_r)$. 
Let $y$ be an element of $P^2$ from this intersection. 

Next, we show that $y\leq x$ in $P^2$.
Since $x$ is a red exit point, we have  $x\in\D_{P^1}(B_k^1)$, $x \in \Up_{P^1}(A_k^1)$, and $x \in\Up_{P^1}(A_{k-1}^1)$ by~\eqref{eq:red-exit-points}. 
Note also that  $y\in\D_{P^1}(B_k^1)$ and $y \in \Up_{P^1}(A_k^1)$ since $y\in P^2$. 
Using Claim~\ref{claim:zigzag-chain}\ref{item:zig-zag-chain} on $P^1$ and zig-zag path $Z^1(x)$, we then see that $y\leq x$ in $P^1$. 
This relation also holds in $P^2$ since $x$ and $y$ are both in $P^2$ (and $P^2$ is an induced subposet of $P^1$).

Now, recall that $x\leq b$ in $P^3$ and hence $x \in \D_{P^3}(B^3)\subseteq \D_{P^2}(B^2_{\ell})$.
Since $y\leq x$ in $P^2$ this implies $y \in \D_{P^2}(B^2_{\ell})$.
Then, from~\eqref{eq:y-points} we deduce that $y\in \Up_{P^2}(\{z_1,z_r\})$ since $y$ is in at least one of $Z^2(y_1)$ and $Z^2(y_r)$.
Therefore, $b\in \Up_{P^2}(\{z_1,z_r\})$ since $y\leq x\leq b$ in $P^2$, which concludes the proof of the claim.
\end{proof}

\subsection{Finishing the proof}
Let $B':=B^3-B^{\mathrm{trapped}}$.
By Claim~\ref{claim:trapped}, we have
\[
\dim(A^3,B^3) \leq \dim(A^3,B') + \dim(A^3,B^{\text{trapped}}) \leq \dim(A^3,B')+12h+6.
\]
It remains to show that $\dim(A^3,B')\leq 12h + 6$, which is the goal of this section.

Let $E^R$ be the set of red exit edges of elements $b\in B'$ and let $E^B$ be the set of blue exit edges of elements $b\in B'$.
We know from the previous section that $E^R$ and $E^B$ are disjoint.
Recall that the graph $G^3$ is connected. 
If, in our drawing of $G$, we contract all of $G^3$ into a single vertex $g^3$, the cyclic ordering of the edges around $g^3$ induces a cyclic ordering of the edges in $E^R \cup E^B$.  
Two remarks are in order here: First, cyclic orderings of edges around a given vertex will always be assumed to be taken in clockwise direction; second, for our purposes we must obviously keep parallel edges resulting from the contractions (loops on the other hand will not be important and can be safely removed).  
It follows from the definition of $B'$ that all edges in $E^R$ appear consecutively in the cyclic ordering, and the same is true for $E^B$.   
Each edge in $E^R$ is incident to a unique red exit point; let $\RED$ denote the set of these red exit points. 
Similarly, let $\BLUE$ be the set of blue exit points incident to edges in $E^B$. 
Let us point out that the two sets $\RED$ and $\BLUE$ are not necessarily disjoint. 
As is easily checked, every element that appears in both sets is a cutvertex of $G^3$.

For each element $v \in  \Up_{P^3}(\RED) - \RED$, choose an element $v'\in P^3$ such that there is $x\in \RED$ with $x \leq v' < v$ in $P^3$ and $v' < v$ is a cover relation in $P^3$. 
The \emph{red path} of element $v\in \Up_{P^3}(\RED)$ is $v$ itself if $v\in \RED$, or if $v\notin \RED$ then it is the path $v_0,\ldots,v_m$ such that $v_0\in \RED$, $v_m=v$ and $v_i$ is the element with $v_i < v_{i+1}$ in $P^3$ chosen for $v_{i+1}$ for each $i\in \{0, \dots, m-1\}$.  
Observe that every two red paths intersect in a (possibly empty) common prefix starting at their endpoint in $\RED$. 
Indeed, if two red paths have a vertex $v$ in common, then for both paths the section between $v$ and their endpoint in $\RED$ is the red path of $v$.  
Replacing $\RED$ with $\BLUE$ in the above definition, we similarly define \emph{blue paths} for all elements $v \in  \Up_{P^3}(\BLUE)$. 
Here also, every two blue paths intersect in a (possibly empty) common prefix starting at their endpoint in $\BLUE$.

For each $b\in B'$, we define a path $Q(b)$ in $G^3$. 
A key property of $Q(b)$ will be that all its elements are contained in the downset of $b$ in $P^3$. 
To define $Q(b)$ we first need to fix a particular element $\m(b)$ in the downset of $b$ as follows:  
Let $\m(b)$ be an arbitrarily chosen minimal element of the subposet of $P^3$ induced by $\D_{P^3}(b)\cap \Up_{P^3}(\RED)\cap \Up_{P^3}(\BLUE)$.  
Note that the latter set is not empty since it contains at least $b$.  
Next, let $\Red(b)$ denote the red path of $\m(b)$, and let $\Blue(b)$ denote the blue path of $\m(b)$.  
Let $Q(b)$ be the union of these two paths. 
Observe that $\m(b)$ is the only vertex of $Q(b)$ in $\Up_{P^3}(\RED)\cap \Up_{P^3}(\BLUE)$. 
Indeed, otherwise $\m(b)$ would not be minimal in $\D_{P^3}(b)\cap \Up_{P^3}(\RED)\cap \Up_{P^3}(\BLUE)$.  
Hence, $Q(b)$ is a path.  
Finally, let $\redv(b)$ be the endpoint of $\Red(b)$ in $\RED$ and let $\rede(b)$ denote the red exit edge of $\redv(b)$. 
Similarly, let $\bluev(b)$ be the endpoint of $\Blue(b)$ in $\BLUE$ and let $\bluee(b)$ denote the blue exit edge of $\bluev(b)$. 

\begin{figure}[t]
	\centering
	\includegraphics[scale=1.0]{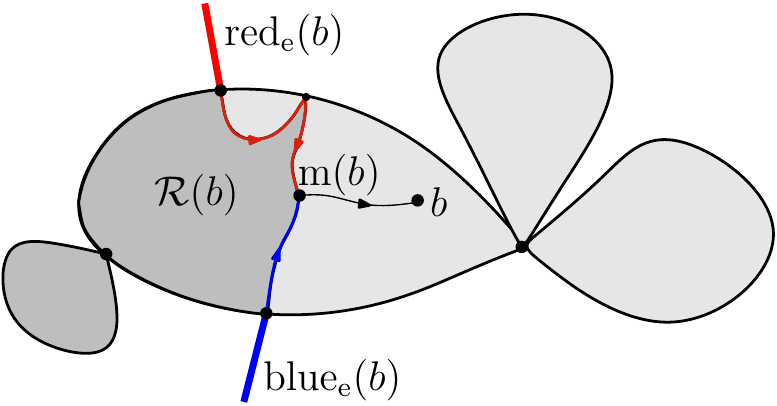}
	\caption{The red path $\Red(b)$ and the blue path $\Blue(b)$.}
	\label{fig:splitting-paths}
\end{figure}

As illustrated in Figure~\ref{fig:splitting-paths}, the path $Q(b)$ defines a corresponding region $\region(b)$ of the plane, which we define precisely now.   
Let $F(b)$ denote the (clockwise) facial trail around the boundary of the outer face of $G^3$ that starts at the first edge of the boundary after $\bluee(b)$ in the cyclic ordering around vertex $\bluev(b)$, and ends at the last edge of the boundary that is before $\rede(b)$ around vertex $\redv(b)$. 
Let then $\region(b)$ be the region of the plane obtained by taking the union of all the bounded faces of the plane graph $F(b) \cup Q(b)$, boundaries included. 
The following observation will be useful: 
\begin{align}
	\label{eq:region}
	\begin{split}
		& \text{If $u, v$ are two vertices of $G^3$ with $u$ contained in $\region(b)$ but not $v$} \\
		& \text{then the path $Q(b)$ separates $u$ from $v$ in $G^3$.} 
	\end{split}
\end{align}
To see this it suffices to observe that any path $P$ from $v$ to $u$ in $G^3$ starts in the outer face of the plane graph $F(b) \cup Q(b)$ and ends in $\region(b)$, thus the first intersection of $P$ with $\region(b)$ (seen from $v$) is a vertex on the boundary of $\region(b)$.  
This vertex must be in $Q(b)$ since all vertices in $\region(b)-Q(b)$ have all their incident edges drawn in $\region(b)$.

\begin{claim}\label{claim:ordering-of-exit-edges}
	Let $b, b' \in B'$. 
	Then the four edges $\rede(b),  \bluee(b),  \rede(b'), \bluee(b')$ 
	are ordered 
	\begin{align*}
		& \bluee(b), \bluee(b'),\rede(b'), \rede(b), & \text{if $b' \in \region(b)$,}\\
		&\bluee(b),\rede(b),\rede(b'), \bluee(b'), & \text{if $b' \not\in \region(b)$,}
	\end{align*}
	in the cyclic ordering of edges around $g^3$, where possibly $\rede(b) = \rede(b')$ or $\bluee(b) = \bluee(b')$. 
\end{claim}
\begin{proof}
	First observe that if $\m(b)=\m(b')$, then $Q(b)=Q(b')$ and hence $\rede(b)=\rede(b')$, $\bluee(b)=\bluee(b')$ so the claim holds vacuously.
	So assume for the rest of the proof that $\m(b)\neq \m(b')$.
	
	Now we show that either both $b'$ and $\m(b')$ are in $\region(b)$ or none of them are.
	Arguing by contradiction, suppose that one is in but not the other. 
	Then $Q(b)$ separates $b'$ from $\m(b')$ in $G^3$ (by \eqref{eq:region}).
	Consider a path $Q$ of $G^3$ witnessing the comparability $\m(b')\leq b'$ in $P^3$.
	Clearly, $Q(b)$ and $Q$ must intersect.
	Let $v$ be a vertex in their intersection.  
	We have $\m(b')\leq v \leq \m(b)$ in $P^3$.
	Recall that each of $\m(b')$ and $\m(b)$ is in both the upset of $\RED$ and the upset of $\BLUE$, and is minimal with this property in $P^3$.
	It follows that $\m(b')=\m(b)$, contradicting our assumption.	
	Therefore, either both $b'$ and $\m(b')$ are in $\region(b)$, or none of them are, as claimed. 
	
	Next, we study the positions of vertices $\redv(b')$ and $\bluev(b')$ depending on the position of $\m(b')$: 
	\begin{enumerate}
		\item If $\m(b')\in\region(b)$ then either $\rede(b)=\rede(b')$ or $\redv(b')$ appears on $F(b)$ but not in $Q(b)$.\label{item:redv-in}
		\item If $\m(b')\in\region(b)$ then either $\bluee(b)=\bluee(b')$ or $\bluev(b')$  appears on $F(b)$ but not in $Q(b)$.\label{item:bluev-in} 
		\item If $\m(b')\not\in\region(b)$ then either $\rede(b)=\rede(b')$ or $\redv(b')$ does not appear on $F(b)$.\label{item:redv-out}
		\item If $\m(b')\not\in\region(b)$ then either $\bluee(b)=\bluee(b')$ or $\bluev(b')$ does not appear on $F(b)$. \label{item:bluev-out}
	\end{enumerate}
	
	We prove \ref{item:redv-in} and \ref{item:redv-out}, properties \ref{item:bluev-in} and \ref{item:bluev-out} follow by a symmetric argument (exchanging red and blue).
	Suppose that $\rede(b)\neq\rede(b')$.
	We first show that $\Red(b')$ and $Q(b)$ are disjoint in this case. 
	If $\Red(b')$ intersects $\Red(b)$ then these two paths have a non-empty common prefix starting from an element in $X'$, thus in particular $\redv(b)=\redv(b')$ and $\rede(b)=\rede(b')$, contrary to our assumption.
	On the other hand, if $\Red(b')$ intersects $\Blue(b)-\m(b)$, then letting $w$ be an element in the intersection we have $\redv(b') \leq w$ and $\bluev(b) \leq w < \m(b)$ in $P^3$, contradicting the minimality of $\m(b)$.
	Thus, $\Red(b')$ and $Q(b)$ are disjoint, as claimed.  
	
	If $\m(b')\in\region(b)$ then by \eqref{eq:region} the whole path $\Red(b')$ is in $\region(b)$ (here we use that $\Red(b')$ and $Q(b)$ are disjoint).
	Since the other endpoint of $\Red(b')$, namely $\redv(b')$, lies on the outer face of $G^3$, it must appear on $F(b)$.
	This shows~\ref{item:redv-in}. 
	
	If $\m(b')\not\in\region(b)$ then by \eqref{eq:region} the whole path $\Red(b')$ lies outside of $\region(b)$ (using again that $\Red(b')$ and $Q(b)$ are disjoint). 
	In particular, $\redv(b')$ does not appear on $F(b)$.  
	This shows~\ref{item:redv-out}. 
	
	Now that properties~\ref{item:redv-in}-\ref{item:bluev-out} have been established, we can proceed with the proof of the claim.  
	First let us consider the case that $b'\in\region(b)$. 	
	Suppose that $\rede(b) \neq \rede(b')$.
	We want to show that the edges $\bluee(b), \rede(b'), \rede(b)$ appear in this order in the cyclic ordering around $g^3$ clockwise.
	We already saw that $b'\in\region(b)$ implies $\m(b')\in\region(b)$.
	Since $\rede(b) \neq \rede(b')$ is assumed,  $\redv(b')$ appears on $F(b)$ but not in $Q(b)$ by~\ref{item:redv-in}. 
	Let $e, f$ be the two edges incident to $\redv(b')$ such that the edges $e, \rede(b'), f$ appear in this order in the ordering around $\redv(b')$ in $G$, and such that $e, f$ are consecutive edges of the boundary of $G^3$ (i.e.\ if we start a facial trail on the outer face at edge $e$, then $f$ is the next edge). 
	Since $Q(b)$ separates vertices in $\region(b)$ from vertices not in $\region(b)$, and since $\redv(b')$ is not in $Q(b)$, we deduce that $e, f$ appear in $F(b)$ (consecutively). 
	By the definition of $F(b)$, this implies that $\bluee(b), \rede(b'),\rede(b)$ appear in this order in the ordering of edges around $g^3$. 
	(One way to see this is as follows: One could start contracting $G^3$ by first contracting $F(b)$ onto vertex $\redv(b')$; then $\bluee(b), \rede(b'),\rede(b)$ appear in this order around  $\redv(b')$, and this ordering remains the same when contracting the rest of $G^3$.)   
	
	A symmetric argument shows that if $b'\in\region(b)$ and $\bluee(b) \neq \bluee(b')$ then $\bluee(b), \bluee(b'),\rede(b)$ appear in this order in the ordering of edges around $g^3$. 
	Now, if $b'\in\region(b)$, $\rede(b) \neq \rede(b')$ and $\bluee(b) \neq \bluee(b')$, then combining these observations with the fact that edges in $E^R$ appear consecutively in the ordering of $E^R \cup E^B$ around $g^3$, we deduce that the only possibility is that $\bluee(b), \bluee(b'),\rede(b'), \rede(b)$ appear in this order in the cyclic ordering. 
	This conclude the proof in the case that $b'\in\region(b)$. 
	
	Next, let us assume that  $b'\notin\region(b)$. 
	Then $\m(b')\notin\region(b)$. 
	If $\rede(b) \neq \rede(b')$, then by~\ref{item:redv-out} and the definition of $F(b)$, we see that $\bluee(b), \rede(b), \rede(b')$ appear in this order around $g^3$. 
	Similarly, if $\bluee(b) \neq \bluee(b')$, then $\bluee(b), \rede(b), \bluee(b')$ appear in this order around $g^3$ by~\ref{item:bluev-out}. 
	The claimed ordering on the edges $\bluee(b),\rede(b),\rede(b'), \bluee(b')$ already follows from these observations in case $\rede(b) = \rede(b')$ or $\bluee(b) = \bluee(b')$. 
	Thus, it only remains to consider the case that $\rede(b) \neq \rede(b')$ and $\bluee(b) \neq \bluee(b')$. 
	Using the previous two observations and the fact that edges in $E^R$ appear consecutively in the ordering of $E^R \cup E^B$ around $g^3$, we deduce that the only possibility is that $\bluee(b),\rede(b),\rede(b'), \bluee(b')$ appear in this order in the cyclic ordering. 
\end{proof}

Thanks to Claim~\ref{claim:ordering-of-exit-edges}, we may weakly order the elements of $B'$ as follows. 
For $b,b'\in B'$ let us write $b \prec b'$ if the edges $\bluee(b), \rede(b),  \rede(b'), \bluee(b')$ are ordered this way around $g^3$ and $|\set{\bluee(b), \rede(b),  \rede(b'), \bluee(b')}|\geq3$.
By Claim~\ref{claim:ordering-of-exit-edges} we know that
\begin{equation}\label{eq:prec-ordering}
b \prec b' \quad\text{or}\quad b'\prec b \quad\text{or}\quad (\rede(b)=\rede(b') \text{ and } \bluee(b)=\bluee(b')),
\end{equation}
for all $b,b'\in B'$.

We are now ready to partition $\Inc(A^3,B')$ into $12h+6$ of reversible sets. 
First we partition $\Inc(A^3,B')$ into $I^{\operatorname{in}}$ and $I^{\operatorname{out}}$. 
The set $I^{\operatorname{in}}$ contains the pairs $(a,b)\in \Inc(A^3,B')$ with $a$ drawn in the region $\region(b)$, while $I^{\operatorname{out}}$ contains the remaining ones. 

\begin{claim}
	$I^{\operatorname{in}}$ and $I^{\operatorname{out}}$ can each be partitioned into at most $6h+3$ reversible sets.
\end{claim}
\begin{proof}
	For each $x\in P^3$ let $B_x:=\set{b\in B'\mid \redv(b)=x}$ and $I_x:=\set{(a,b)\in I^{\operatorname{in}} \mid b\in B_x}$.
	Since $B_x\subseteq \Up_{P^3}(x)$ we may apply our main lemma, Lemma~\ref{lemma:point-below-max}, on $P^3$ with the set $B_x$ and element $x$ to obtain a partition of $I_x$ into at most $6h+3$ reversible sets. 
	Let $I_x^1, \dots, I_x^{6h+3}$ denote such a partition (with possibly some empty sets). 
	Next, for each $i\in \{1, \dots, 6h+3\}$ let $I^i:= \bigcup_{b \in B'} I_{\redv(b)}^i$. 
	Observe that $I^1, \dots, I^{6h+3}$ is a partition of $I^{\operatorname{in}}$ (again, with possibly some empty sets).  
	We claim that each set in this partition is reversible. 
	
	Arguing by contradiction, suppose that  $I^{i}$  is not reversible for some $i\in \{1, \dots, 6h+3\}$. 
	Then there is a sequence of pairs $(a_1,b_1),\dots,(a_k,b_k)$ in $I^{i}$ forming an alternating cycle in $P^3$.  
	Consider an index $j\in \set{1,\ldots,k}$ and let $Q$ be a path witnessing the relation $a_j\leq b_{j+1}$ in $P^3$ (indices are taken cyclically, as always).
	This path cannot intersect $Q(b_j)$ since otherwise an element $v$ in the intersection would witness $a_j \leq v \leq \m(b_j) \leq b_j$ in $P^3$, which cannot be.
	Since $a_j\in\region(b_j)$ and since $Q(b_j)$ separates vertices in $\region(b_j)$ from vertices outside $\region(b_j)$, we conclude that $b_{j+1}\in\region(b_j)$.
	Now, by Claim~\ref{claim:ordering-of-exit-edges}, 
	\[
	b_{j+1} \prec b_j\quad\text{or}\quad (\rede(b_j)=\rede(b_{j+1}) \text{ and } \bluee(b_j)=\bluee(b_{j+1})).
	\]
	Since the above conclusion holds for every $j\in\set{1,\ldots,k}$, we end up with the only possibility that
	\[
	\rede(b_1)=\cdots=\rede(b_k)\quad\text{and}\quad \bluee(b_1)=\cdots=\bluee(b_k).
	\]
	In particular, $\redv(b_1)=\cdots=\redv(b_k)$. 
	Hence,  $(a_1,b_1), \dots, (a_k,b_k) \in I_{\redv(b_1)}^i$. 
	However, this contradicts the fact that $I_{\redv(b_1)}^i$ is reversible. 
	
	The proof that $I^{\operatorname{out}}$ can be partitioned into at most $6h+3$ reversible sets is symmetric.   
\end{proof}

We conclude that $\Inc(A^3,B')$ can be partitioned into at most $12h+6$ reversible sets, as claimed at the beginning of this section.  
Putting everything together, we deduce that 
\[
\dim(P)\leq 8 \dim(A^3,B^3)\leq 8 (24h+12)=192h+96, 
\]
which concludes the proof of our theorem.

\section{New constructions}\label{sec:lower-bounds}
As mentioned in the introduction, the original construction of Kelly~\cite{Kel81} shows that planar posets of height $h$ can have dimension at least $h+1$. 
It was even suggested in~\cite{ST14} (perhaps provocatively) that there might be some constant $c$ such that the dimension of planar posets is at most $h+c$.
Theorem~\ref{thm:lower-bound-1} shows that such a constant cannot exist. 
To prove Theorem~\ref{thm:lower-bound-1}, we prove the following statement that clearly implies it. 

\begin{figure}[t]
 \centering
 \includegraphics[scale=1.0]{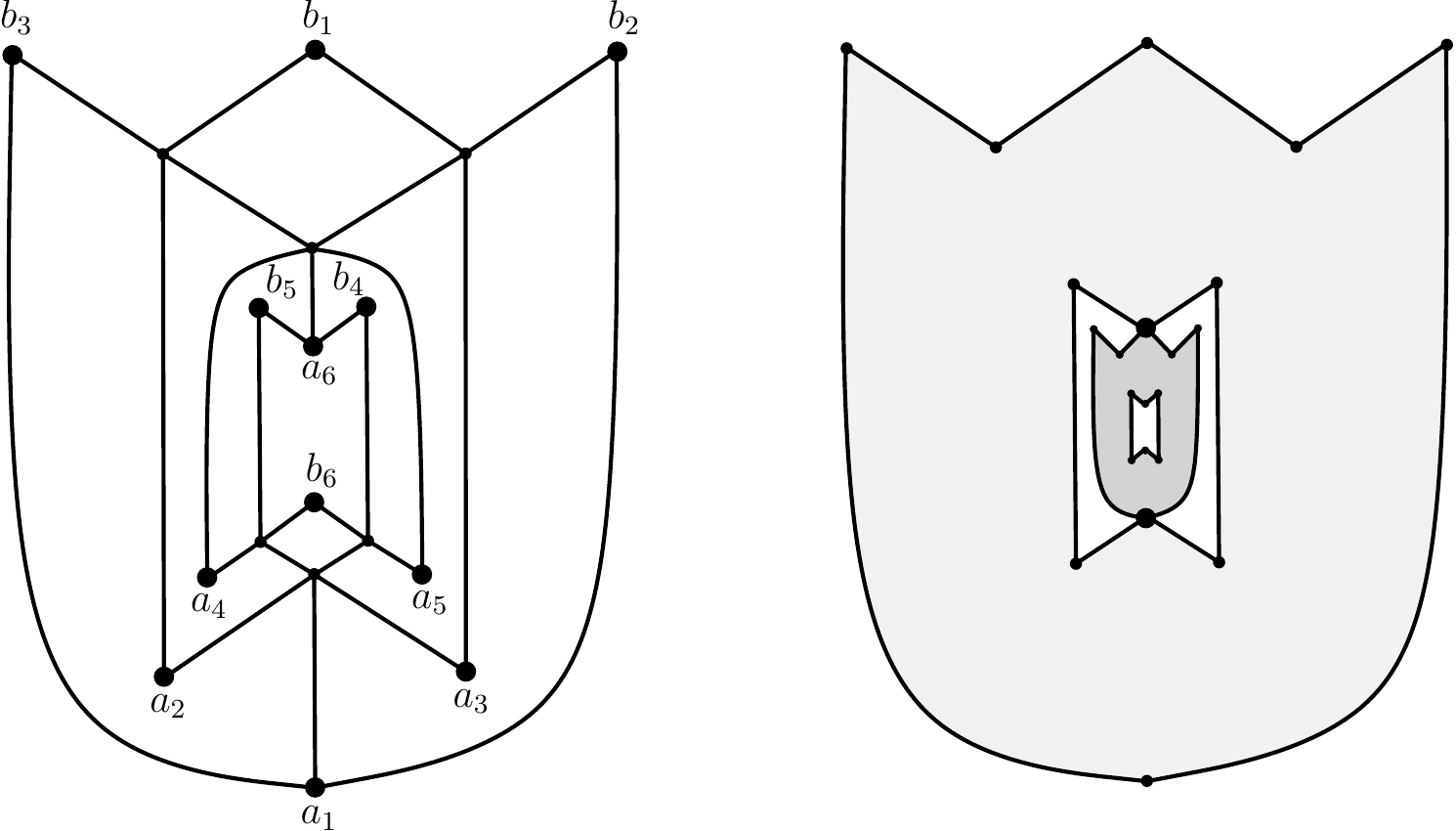}
 \caption{Iterative construction of planar posets with arbitrarily large dimension.}
 \label{fig:lower-bound}
\end{figure}

\begin{theorem}
	For every $h\geq 1$ with $h\equiv 1\mod 3$, there is a planar poset $P$ of height $h$ with
	\[
	\dim(P) \geq (4/3)h+2/3.
	\]
\end{theorem}

\begin{proof}
 If $h=1$ then it suffices to take an antichain of size at least $2$, which has dimension $2$. 
 For larger heights $h$ with $h\equiv 1\mod 3$, we give an inductive construction that contains the standard example of size $(4/3)h+2/3$ as an induced subposet, thus showing that the dimension is at least $(4/3)h+2/3$.
 For the base case $h=4$, start with the planar poset $P$ in Figure~\ref{fig:lower-bound} on the left.
 With the provided labeling it is easy to see that it contains the standard example $S_6$ as an induced subposet.

Next, using this base case, we describe how to build the desired poset for height $h=7$.
Observe that in the diagram of $P$ the maximal element $b_6$ is drawn below the minimal element $a_6$, and that there is some free space between them.  
Thus, we can take a small copy $P'$ of $P$ and insert $P'$ in that space. 
 We do it in such a way that the element $a_1$ of $P'$ is identified with the element $b_6$ of $P$, and the element $b_1$ of $P'$ is identified with element $a_6$ of $P$, see the right of Figure~\ref{fig:lower-bound} for an illustration. 
 The resulting poset has height $7$ and the labeled elements of $P$ and $P'$, except for the ones that we identified, form a standard example of order $10$.

 It is not hard see that this copy-pasting procedure gives rise to an iterative construction, where in each step we increase the height by $3$ and the order of the standard example under consideration by $4$, as desired.  
\end{proof}

Next, we show that there are posets of height $h$ with planar cover graphs and dimension at least $2h-2$, as stated in Theorem~\ref{thm:lower-bound-2}. 

\begin{proof}[Proof of Theorem~\ref{thm:lower-bound-2}]
 Consider the construction of a poset with a planar cover graph illustrated in Figure~\ref{fig:spidernet}.
 Directed edges indicate the precise cover relations. 
 Undirected edges are part of one of the two induced subposets that resemble a spider net (c.f.~Figure~\ref{fig:trotter-spider} in the introduction). 
 These edges are oriented away from the centers of the spider nets.
 The resulting poset has height $6$ and contains a standard example of order $10$ as indicated by the labeled elements.
 Clearly, we can extend this example in such a way that the poset has height $h$ and contains a standard example of order $2h-2$, and hence has dimension at least $2h-2$.
\end{proof}

\begin{figure}[t]
 \centering
 \includegraphics[scale=1.3]{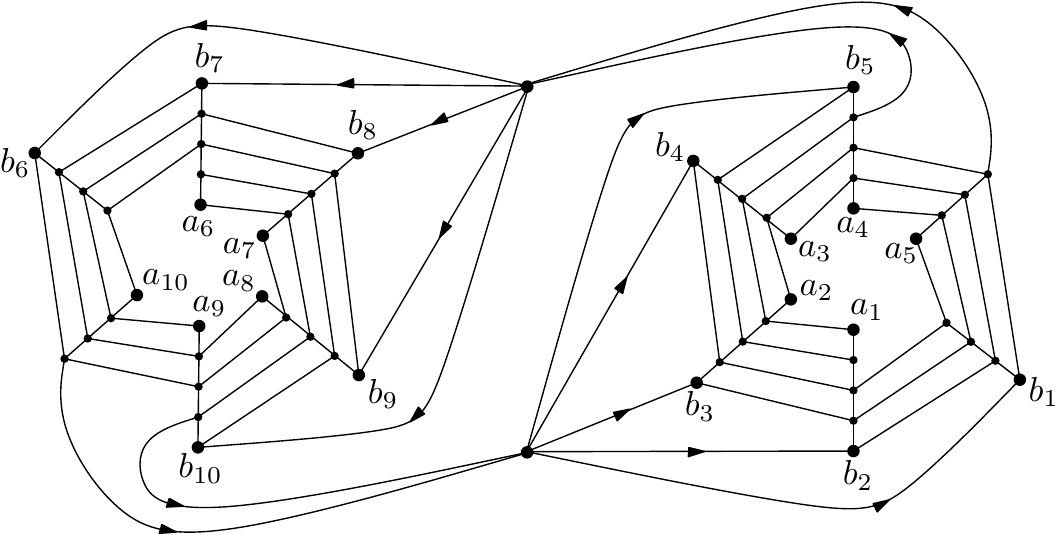}
 \caption{Construction of posets with planar cover graphs and large dimension.}
 \label{fig:spidernet}
\end{figure}

We conclude the paper with the proof of Theorem~\ref{thm:lower-bound-K5}, asserting the existence of posets of height $h$ and dimension at least $2^{h/2}$ whose cover graphs are $K_5$-minor free. 
In fact, we prove something slightly stronger: There are posets  of height $h$ and dimension at least $2^{h/2}$ whose cover graphs have treewidth at most $3$. 
Let us recall first the necessary definitions: 
A {\em tree decomposition} of a graph $G$ consists of a tree $T$ and a collection of non-empty subtrees $T_v$ ($v \in V(G)$) of $T$ such that $T_v$ and $T_w$ intersect for all $vw\in E(G)$. 
The {\em width} of such a decomposition is $\max_{x\in V(T)} |\{v \in V(G): x\in V(T_v)\}| -1$. 
The {\em treewidth} of $G$ is the minimum width of a tree decomposition of $G$. Since graphs with treewidth $t$ have no $K_{t+2}$ minors, the following theorem implies Theorem~\ref{thm:lower-bound-K5}. 

\begin{theorem}\label{thm:lower-bound-tw}
	For each even $h\geq 2$, there is a poset of height $h$ with dimension at least $2^{h/2}$ whose cover graph has treewidth at most $3$. 
\end{theorem}

\begin{proof}
For each even $h \geq 2$ we construct, by induction on $h$, a poset $P_{h}$ of height $h$ whose cover graph has treewidth at most $3$ such that $\Min(P_h) \cup \Max(P_h)$ induces a standard example of order $2^{h/2}=:c_h$.  

For the base case of the induction ($h=2$), it suffices to let $P_2$ be the standard example $S_2$.   
For the inductive step ($h \geq 4$), let $\Min(P_{h-2})=\{a_1, \dots, a_{c_{h-2}}\}$ and $\Max(P_{h-2})=\{b_1, \dots, b_{c_{h-2}}\}$ in such a way that $a_i$ and $b_i$ are incomparable in $P_{h-2}$ for each $i\in \{1, \dots, c_{h-2}\}$.  
Then, the poset $P_{h}$ is obtained from $P_{h-2}$ by introducing for each $i$ four new elements $a^1_i,a^2_i,b^1_i,b^2_i$ forming a standard example of order $2$ (with relations $a^1_i < b^2_i$ and $a^2_i < b^1_i$), and adding the relations $a^1_i < a_i$ and $a^2_i < a_i$, and $b_i < b^1_i$ and $b_i < b^2_i$ (and all relations implied by transitivity). 
The construction is sketched in Figure~\ref{fig:height-tw-ex}.  
Observe that 
\begin{align*}
\Min(P_h) &= \{a^1_1, a^2_1, \dots, a^1_{c_{h-2}}, a^2_{c_{h-2}}\} \\
\Max(P_h) &= \{b^1_1, b^2_1, \dots, b^1_{c_{h-2}}, b^2_{c_{h-2}}\}
\end{align*}
and that these two sets together induce a standard example of order $2c_{h-2}=2^{h/2}=c_h$. 
Moreover, the height of $P_{h}$ is that of $P_{h-2}$ plus $2$, and thus $P_{h}$ has height $h$.

Finally, we show that the cover graph of $P_h$ has treewidth at most $3$. 
This is done by induction on $h$ again. 
To aid the induction, we prove that there exists a tree decomposition of width $3$ with the extra property that, for each incomparable pair $(a, b)$ with $a\in \Min(P_h)$ and $b\in \Max(P_h)$, the subtrees associated to $a$ and $b$ intersect. 

Clearly, this can be done for $h=2$, so let us consider the inductive case ($h \geq 4$). 
Using the same notations as above for the minimal and maximal elements of $P_{h-2}$ and of $P_h$, we extend the tree decomposition $(T, \{T_v\}_{v \in P_{h-2}})$ of the cover graph of $P_{h-2}$ that we obtain by induction as follows: 
For each $i\in \{1, \dots, c_{h-2}\}$, add three new nodes $x_i, y_i, z_i$ to the tree $T$ and the edges $x_iy_i, x_iz_i$, and $x_ir_i$ where $r_i$ is a node of $T$ where the subtrees $T_{a_i}$ and $T_{b_i}$ meet. 
Next, extend the subtree $T_{a_i}$ so that it contains the three new nodes, and extend $T_{b_i}$ with the node $x_i$ only. 
Then, define the new subtrees $T_{a^1_i},T_{a^2_i},T_{b^1_i},T_{b^2_i}$ to be the subtrees induced by $\{y_i\}, \{z_i\}, \{x_i, y_i, z_i\}, \{x_i, y_i, z_i\}$, respectively. 
It is easy to check that the width of the new tree decomposition is still $3$, and moreover that it has the extra property that $T_{a^j_i}$ and $T_{b^j_i}$ intersect for $j=1,2$, as desired.     
\end{proof}

Contrasting with the above theorem, we note that posets with cover graphs of treewidth at most $2$ have dimension bounded from above by an absolute constant (at most $1276$, see~\cite{JMTWW}).

\begin{figure}[t]
 \centering
 \includegraphics[width=0.5\textwidth]{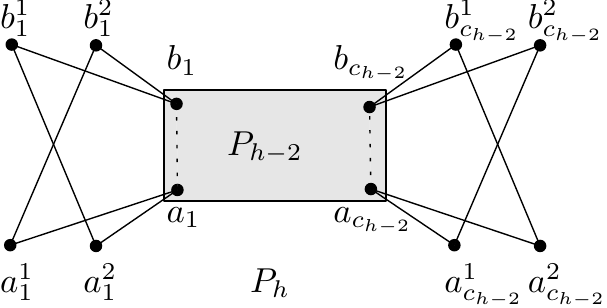}
 \caption{Inductive construction of $P_{h}$.}
 \label{fig:height-tw-ex}
\end{figure}

\section*{Acknowledgements} 
We thank the two anonymous referees for the time and efforts they put in reading this long paper.  
We are also grateful to Micha{\l} Seweryn for his careful reading and his insightful remarks. 

 \bibliographystyle{plain}
 \bibliography{posets-dimension}

\end{document}